\documentclass[11pt]{article}
\usepackage[top=2.5cm,bottom=2.5cm,left=2.5cm,right=2.5cm]{geometry}
\usepackage{mathrsfs}
\usepackage{pstricks}
\usepackage{amsmath,amssymb,graphicx,amsthm,tabularx,array}
\usepackage{hyperref}
\theoremstyle{plain}
\newtheorem{theorem}{Theorem}[section]

\newtheorem{lemma}[theorem]{Lemma}
\newtheorem{defn}[theorem]{Definition}
\newtheorem{prop}[theorem]{Proposition}
\newtheorem{problem}{Problem}[section]

\theoremstyle{definition}

\newtheorem{remark}[theorem]{Remark}

\newtheorem{other}{}

\title{\bf The matching polynomials and spectral radii  of uniform supertrees\footnote{This work was partially supported by National Natural
Science Foundation of China (Nos. 11561032, 11571222, 11471210), the Jiangxi Science Fund for Distinguished Young Scholars and the funds of the Education Department
 of Jiangxi Province (No. GJJ150345).  }}

\author{Li Su$^{a,b}$, \, Liying Kang$^{a}$\footnote{Corresponding author. E-mail addresses: lykang@shu.edu.cn\,(L. Kang),  suli@jxnu.edu.cn\,(L. Su), lhh@mail.ustc.edu.cn\,(H. Li), efshan@i.shu.edu.cn\,(E. Shan)}, \, Honghai Li$^{b}$, \, Erfang Shan$^{c}$\\[5mm]
\small $^a$ Department of Mathematics, Shanghai University, Shanghai 200444,  China\\
\small   $^b$College of Mathematics and Information Science, Jiangxi Normal University\\
\small  Nanchang, Jiangxi 330022,  China\\
\small   $^c$School of Management, Shanghai University,
Shanghai 200444,  China}
\date{}
\begin{document}
\maketitle
\begin{abstract}
We study matching polynomials of uniform hypergraph and spectral radii  of uniform supertrees. By comparing the matching polynomials of supertrees, we extend Li and Feng's  results on grafting operations on graphs to supertrees. Using the methods of grafting operations on supertrees and comparing matching polynomials of supertrees,  we determine the first    $\lfloor\frac{d}{2}\rfloor+1$  largest spectral radii of  $r$-uniform supertrees with size $m$ and diameter $d$. In addition, the first two smallest spectral radii of  supertrees with size $m$ are determined.

\vspace{3mm}


\vspace{2mm}


\end{abstract}

\section{Introduction}

The ordering of graphs by spectral radius was  proposed by Collatz and Sinogowitz~\cite{CollatzSinogowitz57} in 1957.  Lov\'asz and Pelik\'an~\cite{LovaszPelikan73} investigated the spectral radius of trees and determined the first two largest and smallest spectral radii of trees with given order. Brualdi and Solheid~\cite{BrualdiSolheid} proposed the problem of     bounding the spectral radius of some class of graphs and    characterizing the corresponding extremal graphs.   Since then, many authors studied  the spectral radius of trees with some given parameters, such as  degree, diameter, etc.

A {\em hypergraph} $\mathcal{H}$ is a pair $(V,E)$, where $E\subseteq \mathcal{P}(V)$ and $\mathcal{P}(V)$ stands for the power set of $V$. The elements of $V=V(\mathcal{H})$  are referred to as {\em vertices} and the elements of $E=E(\mathcal{H})$ are called {\em hyperedges} or {\em edges}.
A hypergraph $\mathcal{H}$ is {\em $r$-uniform}  if every edge $e\in E(\mathcal{H})$ contains precisely $r$ vertices.  For a vertex $v\in V$, we denote by $E_v$ the set of edges containing $v$. The cardinality $|E_v|$ is the {\em degree} of $v$, denoted by $deg(v)$.
 A vertex with degree one is called a {\em core vertex}, and a vertex with degree larger than one is called an {\em intersection vertex}.
 If any two edges in $\mathcal{H}$ share at most one vertex, then $\mathcal{H}$ is said to be a {\em linear hypergraph}. In this paper we assume  that  hypergraphs are linear and $r$-uniform.

In a hypergraph $\mathcal{H}$,  two vertices $u$ and $v$ are {\em adjacent} if there is an edge $e$ of $\mathcal{H}$ such
that $\{u,v\}\subseteq e$. A vertex $v$ is said to be {\em incident} to an edge $e$ if $v\in e$.
A {\em walk} of hypergraph $\mathcal{H}$ is defined to be an alternating sequence of vertices and edges
$v_1e_1v_2e_2\cdots v_{\ell}e_{\ell}v_{\ell+1}$ satisfying that both $v_{i}$ and $v_{i+1}$ are
incident to $e_i$ for $1\leqslant i\leqslant\ell$. A walk is called a {\em path} if all vertices
and edges in the walk are distinct. The {\em length} of a path is the number of edges in it.
The walk is {\em  closed} if $v_{l+1}=v_1$. A closed walk is called a {\em  cycle} if all vertices and edges in the walk are distinct.
A
hypergraph $\mathcal{H}$ is called {\em connected} if  any two of its vertices are linked by a path in $\mathcal{H}$. The {\em distance} between two vertices is the length of a shortest path connecting
them. The {\em diameter} of a connected $r$-uniform hypergraph $\mathcal{H}$ is the maximum distance among all
vertices of $\mathcal{H}$.
A hypergraph $\mathcal{H}$  is called {\em acyclic} or a  {\em superforest}  if it contains no cycle. A connected superforest  is called a {\em supertree}.

In \cite{LiShaoQi16}  some transformations on hypergraphs such as  ¡°moving edges¡±  and ¡°edge-releasing¡± were introduced and the first two spectral radii of supertrees on $n$ vertices were characterized.  Yuan et. al~\cite{yuanshaoshan-supertrees-16} further determined the first eight uniform supertrees on $n$ vertices with the largest spectral radii. Xiao et. al \cite{XiaoWanglu-supertree-degreesequ-17}  characterized the unique  uniform supertree with the maximum spectral radius among all  uniform supertrees with a given degree sequence.
Recently, the first two largest spectral radii of uniform supertrees with given diameter were characterized in \cite{XiaoWangDu-supertree-diam-18}.

In this paper,  we  determine the first    $\lfloor\frac{d}{2}\rfloor+1$  largest  spectral radii of supertrees among all $r$-uniform supertrees with size $m$ and diameter $d$  and   the first two smallest spectral radii of supertrees with size $m$.
The structure of the remaining part of the paper is as follows: In Section 2, we give some basic definitions and results for tensor and spectra of hypergraphs. Section 3 extends the theory of matching polynomial from graphs to supertrees. By comparing the matching polynomial of supertrees, we  generalize Li and Feng's results on  grafting operations on graphs  to supertrees in Section 4.  By using the method   of grafting operations on supertrees and comparing matching polynomial of supertrees, we  determine the first    $\lfloor\frac{d}{2}\rfloor+1$ spectral radii of supertrees among all $r$-uniform supertrees with size $m$ and diameter $d$ in Section 5.
In  Section 6, the first two smallest spectral radii of supertrees are determined. We give  closing remarks in the last section.

\section{Preliminaries}

Let $\mathcal{H}=(V, E)$ be an $r$-uniform hypergraph on $n$ vertices. A {\em partial hypergraph}  $\mathcal{H}'=(V', E')$ of  $\mathcal{H}$ is a hypergraph with $V'\subseteq V$ and $E'\subseteq E$. A {\em proper partial hypergraph} $\mathcal{H}'$ of  $\mathcal{H}$  is partial hypergraph of $\mathcal{H}$ with $\mathcal{H}'\neq\mathcal{H}$. For a vertex subset $S\subset V$,  let   $\mathcal{H}-S=(V'',E'')$ be the partial hypergraph of $\mathcal{H}$ satisfying that $V''=V\setminus S$, and for any $e\in E$, if $e\subseteq V''$, then $e\in E''$. When $S=\{v\}$,  $\mathcal{H}-S$ is simply written as $\mathcal{H}-v$. For an edge $e=\{v_1,\ldots,v_t\}\in E(\mathcal{H})$, let $\mathcal{H}\setminus e$ stand for the partial hypergraph of $\mathcal{H}$ obtained by   deletion of the edge $e$ from $\mathcal{H}$, i.e. $\mathcal{H}\setminus e=(V, E\setminus\{e\})$, and     $\mathcal{H}-V(e)$ stand for the partial hypergraph of $\mathcal{H}- \{v_1,\ldots,v_t\}$. Denote by $N_k$ the hypergraph consisting of $k$ isolated vertices.

Let $\mathcal{G}$ and $\mathcal{H}$ be two $r$-uniform hypergraphs, and  $u$ a vertex of $\mathcal{G}$ and $v$  a vertex of $\mathcal{H}$.  Denote by  $\mathcal{G}\cdot\mathcal{H}$  the {\em coalescence} of $\mathcal{G}$ and $\mathcal{H}$, obtained from $\mathcal{G}\cup\mathcal{H}$  by identifying $u$ of $\mathcal{G}$ and $v$  of $\mathcal{H}$ (as a new vertex $w$). That is, $V(\mathcal{G}\cdot\mathcal{H})=V(\mathcal{G}-u)\cup V(\mathcal{H}-v)\cup\{w\}$ and  $E(\mathcal{G}\cdot\mathcal{H})=E(\mathcal{G}-u)\cup E(\mathcal{H}-v)\cup\{e'|~ e'=e\setminus\{u\}\cup\{w\}, e\in E_{u}\}\cup\{e'|~ e'=e\setminus\{v\}\cup\{w\}, e\in E_{v}\}$. $\mathcal{H}$ is also called an {\em attached hypergraph} at $w$ of $\mathcal{G}\cdot\mathcal{H}$.

 Let $G=(V,E)$ be an ordinary graph. For every $r\geq3$, the {\em $r$th power }of $G$, denoted by $G^r$, is an $r$-uniform hypergraph with vertex set $V(G^r)=V\cup(\cup_{e\in E}\{i_{e,1},\ldots, i_{e,r-2}\})$ and edge set
  $E(G^r)=\{e\cup\{i_{e,1},\ldots, i_{e,r-2},\}|~e\in E\}$. The $r$th power of an ordinary tree is called a {\em hypertree} (see \cite{HuQishao13}). Note that all hypertrees are supertrees by the definition.  Let $P_m$ and $S_m$ denote the path and the star with $m$ edges, respectively. The $r$th power of $P_m$ and $S_m$, denoted by $P_{m}^r$ and $S_{m}^r$,  are called {\em loose path} and {\em hyperstar}, respectively.

   Let $\mathcal{H}=(V, E)$ be an $r$-uniform hypergraph.   An edge $e$ is called a {\em pendent edge} if $e$ contains
exactly $r-1$ core vertices. If $e$ is not a pendent edge,  it is  called a
{\em non-pendent edge}.  A path $P=(v_0,e_1,v_1,\ldots, v_{p-1},e_p,v_p)$ of $\mathcal{H}$ is called a {\em pendent path} (attached at $v_0$), if all of the vertices
$v_1,\ldots, v_{p-1}$ are of degree two, the vertex $v_p$   and all the $r-2$ vertices in the set $e_i\setminus\{v_{i-1}, v_i\}$ are core vertices in $\mathcal{H}$ ($i=1,\ldots,p$).

For positive integers $r$ and $n$, a real
{\em tensor} $\mathcal{A}=(a_{i_1i_2\cdots i_r})$ of order $r$ and dimension $n$
refers to a multidimensional array (also called {\em hypermatrix}) with entries
$a_{i_1i_2\cdots i_r}$ such that $a_{i_1i_2\cdots i_r}\in\mathbb{R}$ for
all $i_1$, $i_2$, $\ldots$, $i_r\in[n]$.

The following product of tensors, defined by Shao \cite{Shao-tensorproduct},  is a generalization of
the matrix product.
Let  $\mathcal{A}$ and $\mathcal{B}$ be dimension $n$, order  $r\geqslant 2$
and order $k\geqslant 1$ tensors, respectively. Define the product
$\mathcal{AB}$ to be the tensor $\mathcal{C}$ of  dimension $n$ and order
$(r-1)(k-1)+1$ with entries as
\begin{eqnarray}\label{formu1}
c_{i\alpha_1\cdots\alpha_{r-1}}=\sum_{i_2,\ldots,i_r=1}^na_{ii_2\cdots i_r}
b_{i_2\alpha_1}\cdots b_{i_r\alpha_{r-1}},
\end{eqnarray}
where $i\in [n]$, $\alpha_1,\ldots,\alpha_{r-1}\in [n]^{k-1}$.

From the above definition, if $x=(x_1,x_2,\ldots,x_n)^{\mathrm{T}}\in \mathbb{C}^n$ is a complex
column vector of dimension $n$, then  by (\ref{formu1}) $\mathcal{A}x$ is a vector in $\mathbb{C}^n$ whose $i$th component
is given by
\begin{equation*}
\label{eq:Ax equation}
(\mathcal{A}x)_i=\sum_{i_2,\ldots,i_r=1}^na_{ii_2\cdots i_r}x_{i_2}\cdots x_{i_r},~~
\mbox{for each}\, \,i\in [n].
\end{equation*}

In 2005, Qi \cite{qi05} and Lim
\cite{Lim05} independently introduced the concepts of tensor eigenvalues and the spectra of tensors.

Let $\mathcal{A}$ be an order $r$ dimension $n$ tensor, $x=(x_1,x_2,\ldots,x_n)^{\mathrm{T}}\in\mathbb{C}^n$
 a column vector of dimension $n$. If there exists a number $\lambda\in\mathbb{C}$
and a nonzero vector $x\in\mathbb{C}^{n}$ such that
\begin{equation*}
\mathcal{A}x=\lambda x^{[r-1]},
\end{equation*}
where $x^{[r-1]}$ is a vector with $i$-th entry $x^{r-1}_i$, then $\lambda$ is called an {\em eigenvalue} of $\mathcal{A}$, $x$ is called
an {\em eigenvector} of $\mathcal{A}$ corresponding to the eigenvalue $\lambda$.
The {\em spectral
radius} of $\mathcal{A}$ is the maximum modulus of the eigenvalues of $\mathcal{A}$.

In 2012, Cooper and Dutle \cite{CoopDut12} defined the
adjacency tensors for $r$-uniform hypergraphs.

\begin{defn}{\rm
(\cite{CoopDut12})
Let $\mathcal{H}=(V, E)$ be an $r$-uniform hypergraph on $n$ vertices. The adjacency
tensor of $\mathcal{H}$ is defined as the order $r$ and dimension $n$ tensor
$\mathcal{A}(\mathcal{H})=(a_{i_1i_2\cdots i_r})$, whose $(i_1i_2\cdots i_r)$-entry is
\[
a_{i_1i_2\cdots i_r}=\begin{cases}
\frac{1}{(r-1)!}, & \text{if}~\{i_1,i_2,\ldots,i_r\}\in E,\\
0, & \text{otherwise}.
\end{cases}
\] }
\end{defn}

The {\em spectral radius of hypergraph} $\mathcal{H}$ is defined as spectral radius  of its  adjacency tensor, denoted by $\rho(\mathcal{H})$.
In \cite{Fri} the weak irreducibility of nonnegative tensors was defined. It
was proved that an $r$-uniform hypergraph $\mathcal{H}$ is connected if and only if its adjacency
tensor $\mathcal{A}(\mathcal{H})$ is weakly irreducible (see \cite{Fri} and
\cite{YANGYANG-11}). Part of the Perron-Frobenius theorem for nonnegative tensors is
stated in the following for reference.
\begin{theorem}{\rm (\cite{QiLuo-2017})}
\label{thm:Perron-Frobenius}
Let $\mathcal{A}$ be a nonnegative tensor of order $r$ and dimension $n$, where $r, n\geq2$. Then   $\rho(\mathcal{A})$ is an eigenvalue of $\mathcal{A}$ with a nonnegative eigenvector corresponding to it.
If  $\mathcal{A}$ is weakly irreducible, then $\rho(\mathcal{A})$ is a positive
 eigenvalue of $\mathcal{A}$ with a positive eigenvector ${x}$. Furthermore,  $\rho(\mathcal{A})$ is the unique eigenvalue of  $\mathcal{A}$ with a positive eigenvector, and $x$ is the unique positive eigenvector associated with $\rho(\mathcal{A})$, up to a multiplicative  constant.
\end{theorem}
The unique positive eigenvector $x$ with $\sum_{i=1}^nx_i^r=1$  corresponding to $\rho(\mathcal{H})$ is called the {\em principal eigenvector} of $\mathcal{H}$.

\begin{theorem}{\rm (\cite{YANGYANG-SIAM-10})}
\label{yang}
 Let $\mathcal{A}, \mathcal{B}$ be order $r$ and dimension $n$ nonnegative tensors, and $\mathcal{A}\not=\mathcal{B}$. If $\mathcal{B}\leq \mathcal{A}$ and $\mathcal{A}$ is weakly irreducible, then $\rho(\mathcal{A})>\rho(\mathcal{B})$.
\end{theorem}

The  following result can be obtained directly from Theorem \ref{yang} and will be often used in the sequel.

\begin{theorem}\label{YANGYANG-SIAM-10}
Suppose that $\mathcal{G}$  is a uniform hypergraph,   and $\mathcal{G}'$ is a partial hypergraph of $\mathcal{G}$. Then  $\rho(\mathcal{G}')\leq\rho(\mathcal{G})$. Furthermore, if in addition $\mathcal{G}$ is connected and $\mathcal{G}'$ is a proper partial hypergraph, we have  $\rho(\mathcal{G}')<\rho(\mathcal{G})$.
\end{theorem}

An operation of {\em moving edges} on hypergraphs was introduced by Li et. al in \cite{LiShaoQi16}.
Let $\mathcal{H}=(V,E)$ be a hypergraph with $u\in V$ and $e_1,\ldots, e_k\in E$, such that $u\notin e_i$ for $i=1,\ldots, k$. Suppose that $v_i\in e_i$ and write $e_i'=(e_i\setminus \{v_i\})\cup \{u\}  \  (i=1,\ldots, k)$. Let $\mathcal{H}'=(V,E')$ be the hypergraph with $E'=(E\setminus\{e_1,\ldots, e_k\})\cup \{e_1',\ldots, e_k'\}$. Then we say that $\mathcal{H}'$ is obtained from $\mathcal{H}$ by {\em moving edges} $(e_1,\ldots, e_k)$ from $(v_1,\ldots, v_k)$ to $u$.

 \begin{theorem}{\rm (\cite{LiShaoQi16})}\label{lem-edgemoving}
 Let $\mathcal{H}$ be a connected hypergraph, $\mathcal{H}'$ be the hypergraph obtained from $\mathcal{H}$ by moving edges $(e_1,\ldots, e_k)$  from $(v_1,\ldots, v_k)$ to $u$.  If $x$ is the principal eigenvector of $\mathcal{H}$ corresponding to $\rho(\mathcal{H})$, and suppose that $x_u\ge \max_{1\leq i\leq k}\{x_{v_i}\}$, then $\rho(\mathcal{H}')>\rho(\mathcal{H})$.
\end{theorem}

The following {\em edge-releasing operation} on linear hypergraphs was given in \cite{LiShaoQi16}.

Let $\mathcal{H}$ be an $r$-uniform linear hypergraph,  $e$ be a non-pendent edge of $\mathcal{H}$ and $u\in e$.
  Let $e_1,e_2,\ldots,e_k$ be all  edges of $G$ adjacent to $e$ but not containing $u$, and suppose that $e_i\cap e=\{v_i\}$ for $i=1,\ldots,k$.    Let $\mathcal{H}'$ be the hypergraph obtained from $\mathcal{H}$ by moving edges $(e_{1},\ldots, e_k)$ from $(v_{1},\ldots, v_k)$ to $u$.  Then $\mathcal{H}'$ is said to be obtained by an {\em edge-releasing operation} on $e$ at $u$.

   By the above definition we  see that if  $\mathcal{H}'$ and $\mathcal{H}''$ are the hypergraphs obtained from an $r$-uniform linear hypergraph $\mathcal{H}$ by an edge-releasing operation on some $e$ at $u$ and at $v$, respectively. Then $\mathcal{H}'$ and $\mathcal{H}''$ are isomorphic. So we simply say $\mathcal{H}'$ is obtained from $\mathcal{H}$ by an {\em edge-releasing operation} on $e$.

The following result was obtained by Zhou et.al \cite{ZhouSunWangBu-EJC-14}, we will use it  in the sequel.
 \begin{theorem}{\rm (\cite{ZhouSunWangBu-EJC-14})}\label{lem_eiggraph-eigpowergraph}
If $\lambda\neq0$  is an eigenvalue of a graph $G$, then $\lambda^{\frac{2}{r}}$ is an eigenvalue of $G^{r}$. Moreover, $\rho(G^{r})=\rho(G)^{\frac{2}{r}}$.
\end{theorem}

\section{The matching polynomial of hypergraphs}

Let $\mathcal{H}=(V,E)$ be an   $r$-uniform hypergraph of order $n$ and size $m$.    {\em A matching} of $\mathcal{H}$ is a set of pairwise
nonadjacent edges in $E$. A {\em $k$-matching} is  a matching
consisting of $k$ edges. We denote by $m(\mathcal{H},k)$ the number of
$k$-matchings of $\mathcal{H}$. The  {\em matching number} $\nu(\mathcal{H})$ of $\mathcal{H}$ is the maximum cardinality of a matching.

Recently, Zhang et. al ~\cite{Zhang_17} obtained the following result.

\begin{theorem}{\rm (\cite{Zhang_17})} \label{thm_Zhang_matchingpolyradius}
 $\lambda$ is a nonzero eigenvalue of a supertree $\mathcal{H}$ with the corresponding
eigenvector $x$ having all elements nonzero if and only if it is a root of the polynomial
\begin{equation*}\label{e_matchenergy}
\varphi(\mathcal{H},x)=\sum\limits_{k=0}^{\nu(\mathcal{H})}(-1)^{k}m(\mathcal{H},k)x^{(\nu(\mathcal{H})-k)r}.
\end{equation*}
\end{theorem}

Based on the result above, Clark and Cooper  \cite{ClarkCooper} called the polynomial in Theorem~\ref{thm_Zhang_matchingpolyradius} as matching polynomial of $\mathcal{H}$.    Set $m(\mathcal{H},0)=1$. We redefine the {\em matching polynomial} of $\mathcal{H}$ as
\begin{equation*}
\varphi(\mathcal{H},x)=\sum\limits_{k\geq0}(-1)^{k}m(\mathcal{H},k)x^{n-kr}.
\end{equation*}
For exmaple, the matching polynomial of $N_k$ is $\varphi(N_k,x)=x^k$, rather than 1 by Zhang's definition.
The definition here seems more appropriate as it guarantees that  matching polynomials of hypergraphs of the same order have the same degree and the result in Theorem~\ref{thm_Zhang_matchingpolyradius} is still valid.

Some classical results on matching polynomial of a graph can be extended to  a hypergraph as well. However, the matching polynomial of a hypergraph
 has its own flavour, e.g. as shown  in \cite{ClarkCooper}, the roots of matching polynomial of  an $r$-uniform hypergraph with $r>2$ need not necessarily be real.

\begin{theorem}\label{thm_matchingpoly}
Let $\mathcal{G}$ and $\mathcal{H}$ be two $r$-uniform hypergraphs. Then the following statements hold.
\begin{enumerate}
  \item $\varphi(\mathcal{G}\cup\mathcal{H},x)=\varphi(\mathcal{G},x)\varphi(\mathcal{H},x)$.
\item $\varphi(\mathcal{G},x)=\varphi(\mathcal{G}\setminus e,x)-\varphi(\mathcal{G}-V(e),x)$ if $e$ is an edge of $\mathcal{G}$.
\item If $u\in V(\mathcal{G})$ and $I=\{i | e_i\in E_u\}$, for any $J\subseteq I$, we have
\begin{eqnarray*}
\varphi(\mathcal{G},x)=\varphi(\mathcal{G}\setminus\{e_i : i\in J\},x)-\sum_{i\in J}\varphi(\mathcal{G}-V(e_i),x)
\end{eqnarray*}
and
\begin{eqnarray*}\varphi(\mathcal{G},x)=x\varphi(\mathcal{G}-u,x)-\sum_{e\in E_u}\varphi(\mathcal{G}-V(e),x).
\end{eqnarray*}
\item $\sum_{u\in V(\mathcal{G})}\varphi(\mathcal{G}-u,x)=\frac{d}{dx}\varphi(\mathcal{G},x)$.
\end{enumerate}
\end{theorem}
\begin{proof}
(a) From the fact that each $k$-matching in $\mathcal{G}\cup\mathcal{H}$ consists of an $s$-matching in $\mathcal{G}$  combined with a  $(k-s)$-matching from $\mathcal{H}$  for some $s$, the result  follows immediately.

(b) In order to compute the matching polynomial, we count the number of $k$-matching in  $\mathcal{G}$ according to the edge $e$ being contained  or not. The number of $k$-matching not containing $e$ is equal to $m(\mathcal{G}-e, k)$. The number of $k$-matching containing $e$ is equal to $m(\mathcal{G}-V(e),k-1)$. Thus we have
\begin{align*}
  m(\mathcal{G}, k)=m(\mathcal{G}-e, k)+m(\mathcal{G}-V(e), k-1).
\end{align*}
By comparing the coefficients of the corresponding matching polynomial in two sides of (b), the result follows.

(c) Assume that $\{e_i\}_{i\in J}=\{e_1,\ldots,e_s\}$.  Applying (b) of Theorem~\ref{thm_matchingpoly}, we have
\begin{align*}
\varphi(\mathcal{G},x)&=\varphi(\mathcal{G}\setminus e_1,x)-\varphi(\mathcal{G}-V(e_1),x)\nonumber\\
&=\varphi(\mathcal{G}\setminus \{e_1,e_2\},x)-\varphi(\mathcal{G}-V(e_2),x)-\varphi(\mathcal{G}-V(e_1),x)
\end{align*}
Repeatedly using (b) of Theorem~\ref{thm_matchingpoly}, we get
\begin{align}\label{e-matchpoly-delet-somedeges}
\varphi(\mathcal{G},x)&=\varphi(\mathcal{G}\setminus\{e_1,e_2,\ldots, e_s\},x)-\sum_{i=1}^s\varphi(\mathcal{G}-V(e_i),x)\nonumber\\
&=\varphi(\mathcal{G}\setminus\{e_i : i\in J\},x)-\sum_{i\in J}\varphi(\mathcal{G}-V(e_i),x).
\end{align}

\noindent
Note that $u$ is an isolated vertex of $\mathcal{G}-\cup_{i\in I}e_i$,   it follows directly from ~\eqref{e-matchpoly-delet-somedeges} that
\begin{align*}
\varphi(\mathcal{G},x)&=x\varphi(\mathcal{G}-u,x)-\sum_{e_i\in E_u(G)}\varphi(\mathcal{G}-V(e_i),x).
\end{align*}

(d) Consider the ordered pairs $(u, M)$, where $M$ is a $k$-matching in $\mathcal{G}$ and $u$ is a vertex of $\mathcal{G}$ not covered by $M$.
 Counting the number of the ordered pairs, we obtain that the number of  such ordered pairs is equal to
$ m(\mathcal{G}, k)(n-rk)$, which is just the absolute value of the coefficient of $x^{n-rk-1}$ in $\frac{d}{dx}\varphi(\mathcal{G},x)$. On the other hand, if we choose a vertex first, say $u$,  then the number of  $k$-matching not covering $u$  is equal to $ m(\mathcal{G}-u, k)$. Then, the number of such ordered pairs is equal to $\sum_{u\in V(\mathcal{G})}m(\mathcal{G}-u, k)$, which is the absolute value of the coefficient of $x^{n-rk-1}$ in $\sum_{u\in V(\mathcal{G})}\varphi(\mathcal{G}-u,x)$.  The desired result follows.
\end{proof}

\begin{prop}\label{prop-tree-powertree-matchpoly}
Let $T$ be an ordinary tree on $n$ vertices,   $r \, (r\geq3)$ a positive integer. Then the matching polynomials of $T$ and its $r$th power $T^r$ satisfy the following relation:
\[
\varphi(T^r,x)=x^{\frac{(n-2)(r-2)}{2}}\varphi(T,x^{\frac{r}{2}}).
\]
\end{prop}
\begin{proof}
It is easy to see that  $m(T,k)=m(T^r,k)$ for any $k$. Let $n'$ denote the order of $T^r$. Then $n'=n+(n-1)(r-2)$. So we have
\begin{align*}
\varphi(T^r,x)&=\sum_{k\geq0}(-1)^km(T^r,k)x^{n'-kr}=\sum_{k\geq0}(-1)^km(T,k)(y^{\frac{2}{r}})^{n'-kr}\\
&=y^{\frac{2n'}{r}-n}\sum_{k\geq0}(-1)^km(T,k)y^{n-2k}=x^{\frac{(n-2)(r-2)}{2}}\varphi(T,x^{\frac{r}{2}}),
\end{align*}
where a new variable $y=x^{\frac{r}{2}}$ is used in the second and third  equations.
\end{proof}

The ordering on forests has been introduced by  Lov\'asz and Pelik\'an in \cite{LovaszPelikan73}. Now we extend the ordering on forests to superforests. Let $\mathcal{T}$ and  $\mathcal{T}'$ be superforests of $n$ vertices.
  We call $\mathcal{T}'\preceq \mathcal{T}$  if $\varphi(\mathcal{T}',x)\geq \varphi(\mathcal{T},x)$ for every $x\geq \rho(\mathcal{T}')$; call $\mathcal{T}'\prec \mathcal{T}$  if  $\mathcal{T}'\preceq \mathcal{T}$ and the polynomial $\varphi(\mathcal{T}',x)-\varphi(\mathcal{T},x)$ does not vanish at the point $x=\rho(\mathcal{T}')$.  Note that $\mathcal{T}'\prec \mathcal{T}$ ($\mathcal{T}'\preceq \mathcal{T}$, resp.) implies $\rho(\mathcal{T}')< \rho(\mathcal{T})$ ($\rho(\mathcal{T}')\leq\rho(\mathcal{T})$, resp.).

 \begin{remark}
From (a) of Theorem~\ref{thm_matchingpoly}, it is easily seen  that if $\mathcal{T}'\preceq \mathcal{T}$ ($\mathcal{T}'\prec  \mathcal{T}$, resp.), then $\mathcal{T}'\cup \mathcal{H}\preceq \mathcal{T}\cup \mathcal{H}$ ($\mathcal{T}'\cup \mathcal{H}\prec  \mathcal{T}\cup \mathcal{H}$, resp.) for any superforest $\mathcal{H}$.
\end{remark}

\section{Grafting transformations on uniform supertrees}

Li and Feng \cite{LiFeng79}  investigated how the spectral radius change when  a certain transformation is applied to the graph, and obtained the following result.

\begin{theorem}{\rm (\cite{LiFeng79})}\label{thm-lifeng79}
Let $u,v$  be two vertices of $G$ such that $d(u,v)=m$. Let $G(u,v;p,q)$ denote the graph obtained from $G$ by attaching a path of length $p$ at $u$ and a path of length $q$ at $v$. Then $\rho(G(u,v;p,q))>\rho(G(u,v;p+1,q-1))$ under any of the following conditions
\begin{enumerate}
  \item $m=0$, $deg(u)\geq1$, and $p\geq q\geq1$;
  \item $m=1$, $deg(u)\geq2$, $deg(v)\geq2$ and $p\geq q\geq1$;
  \item $m>1$, $deg(u)\geq2$, $deg(v)\geq2$, $p-q\geq m$ and $q\geq1$.
\end{enumerate}
\end{theorem}

Since then, the result has been extensively used in spectral perturbation and proved to be efficient in ordering graphs by spectral radius.  The result  above  is proved by comparing characteristic polynomials of graphs. The characteristic polynomial of a hypergraph is  complicated and very little is known about it up to now. However the result of Theorem~\ref{thm_Zhang_matchingpolyradius}  makes it feasible to compare the spectral radii of supertrees by using  the matching polynomials of supertrees.

It is known that for any forest, its matching polynomial and characteristic polynomial coincide. Following a similar proof of Lemma 4 in \cite{LovaszPelikan73}, the following result can be obtained.
\begin{prop}\label{prop_pa-pb-matchpoly}
If $a+b=c+d$, $a<c\leq d$, then  $ P_a\cup P_b\prec P_c\cup P_d$.
\end{prop}

Based on Propositions ~\ref{prop-tree-powertree-matchpoly} and \ref{prop_pa-pb-matchpoly}, the corresponding result for hypertree can be easily obtained.

\begin{prop}\label{prop_powertree-pa-pb-matchpoly}
If $a+b=c+d$, $a<c\leq d$, $r (r\geq3)$ is an integer, then $ P_a^r\cup P_b^r\prec P_c^r\cup P_d^r$.
\end{prop}

\begin{theorem}\label{thm_subgraphmatchpoly}
If $\mathcal{T}$  is an  uniform supertree,   and $\mathcal{T}'$ is a proper partial hypergraph of $\mathcal{T}$ with $V(\mathcal{T}')=V(\mathcal{T})$, then $\mathcal{T}'\prec\mathcal{T}$.
\end{theorem}
\begin{proof}
Without loss of generality, we assume that $\mathcal{T}'=\mathcal{T}\setminus e$ for some $e$ in $\mathcal{T}$. If $x\geq\rho(\mathcal{T}')$, then $x>\rho(\mathcal{T}-V(e))$ by Theorem~\ref{YANGYANG-SIAM-10}. So $\varphi(\mathcal{T}-V(e), x)>0$. Further by Theorem~\ref{thm_matchingpoly},
\begin{align*}
\varphi(\mathcal{T}, x)=\varphi(\mathcal{T}', x)-\varphi(\mathcal{T}-V(e), x)<\varphi(\mathcal{T}', x),
\end{align*}
 the desired result follows.
\end{proof}

\begin{figure}[!hbpt]
\begin{center}
\includegraphics[scale=0.6]{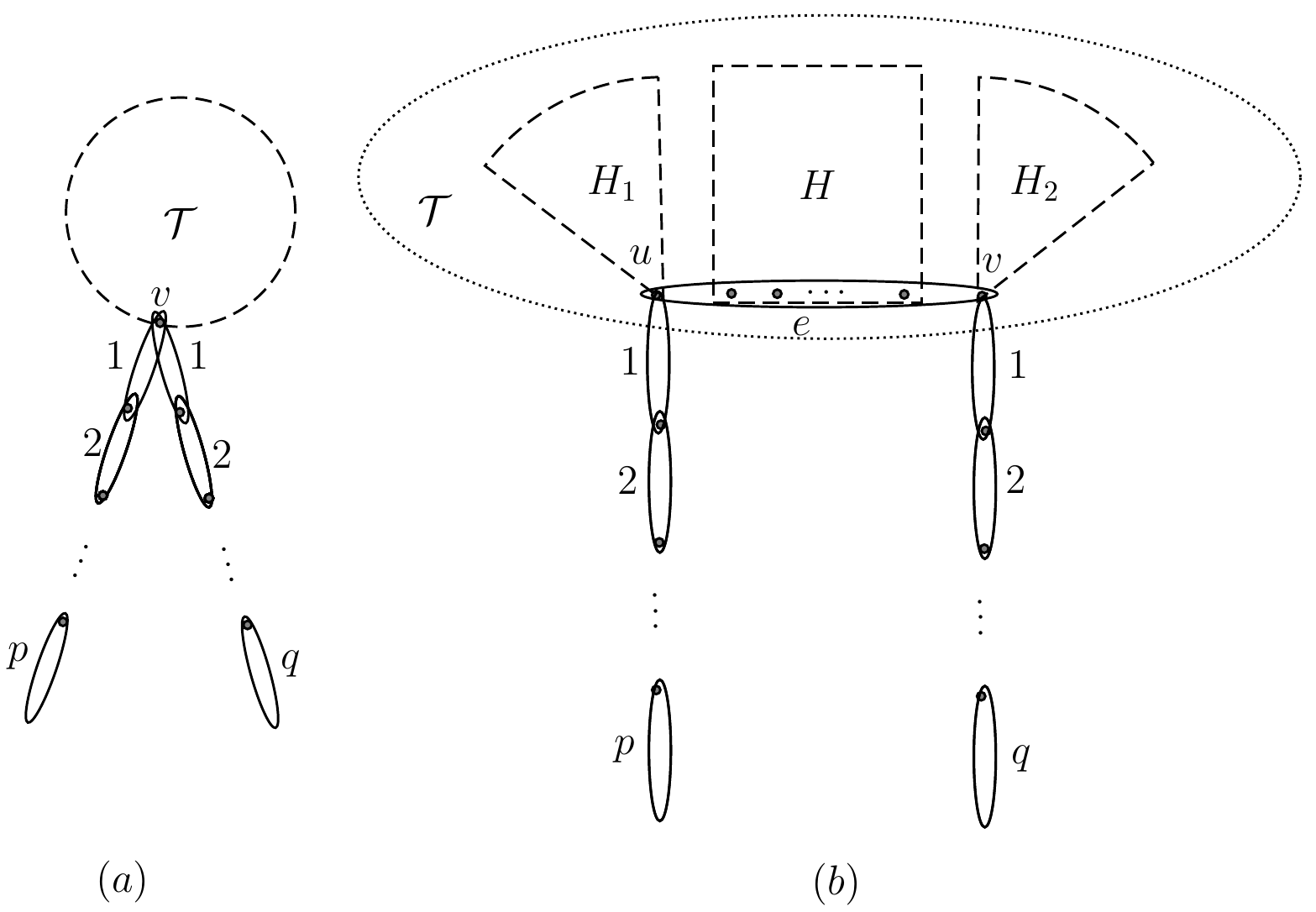}
\caption{Supertrees $(a)\,\,\,\mathcal{T} ( v;p,q)$;\quad$(b)\,\,\,\mathcal{T}^{(1)}(u,v;p,q)$}.
\end{center}\label{figTv43Tuv43}
\end{figure}
Suppose that $\mathcal{T}$  is an  $r$-uniform supertree and $v$ is a vertex in  $\mathcal{T}$.   Let $\mathcal{T}(v;p,q)$ be obtained by attaching two pendent paths of length $p$ and $q$ at $v$ (see  Fig.~1(a)).

\begin{theorem}\label{thm_graftat1vertex}
 If $p\geq q\geq 1$, then $\mathcal{T}(v;p,q)\succ\mathcal{T}(v;p+1,q-1).$
In particular, $\rho(\mathcal{T}(v;p,q))>\rho(\mathcal{T}(v;p+1,q-1)).$
\end{theorem}
\begin{proof}
We first consider the case that $p\geq q=1$. Applying  (b) of Theorem~\ref{thm_matchingpoly} on $\mathcal{T}(v;p,1)$ and   the pendent edge attached at $v$, we have
\begin{align}\label{equ3}
 \varphi(\mathcal{T}(v;p,1),x)=& x^{r-1}\varphi(\mathcal{T}(v;p,0),x)- x^{r-2}\varphi((\mathcal{T}-v)\cup P_{p-1}^r).
\end{align}
\noindent
Similarly, applying  (b) of Theorem~\ref{thm_matchingpoly} on $\mathcal{T}(v;p+1,0)$ and   the pendent edge of the pendent path of length $p+1$  attached at $v$, we have
\begin{align}\label{equ4}
 \varphi(\mathcal{T}(v;p+1, 0),x)=& x^{r-1}\varphi(\mathcal{T}(v;p,0),x)- x^{r-2}\varphi(\mathcal{T}(v;p-1,0)).
\end{align}
By (\ref{equ3}) and (\ref{equ4}), we deduce that
\begin{align*}
 \varphi(\mathcal{T}(v;p,1),x)-\varphi(\mathcal{T}(v;p+1,0),x)
 =x^{(r-2)}(\varphi(\mathcal{T}(v;p-1,0))-\varphi((\mathcal{T}-v)\cup P_{p-1}^r)).
\end{align*}
Note  that $(\mathcal{T}-v)\cup P_{p-1}^r$ is a proper partial hypergraph of $\mathcal{T}(v;p-1,0)$. By Theorems  \ref{YANGYANG-SIAM-10} and \ref{thm_subgraphmatchpoly}, the desired result follows.

When $p\geq q\geq2$, applying (b) of Theorem~\ref{thm_matchingpoly} on $\mathcal{T}(v;p,q)$ and  the pendent edge of the  pendent path of length $q$ attached at $v$, we have
\begin{align}\label{eq3}
  \varphi(\mathcal{T}(v;p,q),x)= & x^{r-1}\varphi(\mathcal{T}(v;p,q-1),x)-x^{r-2}\varphi(\mathcal{T}(v;p,q-2),x).
\end{align}
\noindent
Similarly,
\begin{align}\label{eq4}
\varphi(\mathcal{T}(v;p+1,q-1),x)=& x^{r-1}\varphi(\mathcal{T}(v;p,q-1),x)-x^{r-2}\varphi(\mathcal{T}(v;p-1,q-1),x).
\end{align}

\noindent
By (\ref{eq3}) and (\ref{eq4}), we deduce that
\begin{align*}
    \varphi(\mathcal{T}(v;p,q),x)-\varphi(\mathcal{T}(v;p+1,q-1),x)
   = x^{r-2}(\varphi(\mathcal{T}(v;p-1,q-1),x)-\varphi(\mathcal{T}(v;p,q-2),x)).
\end{align*}

\noindent
Continue this process, we get
\begin{align}\label{e_pqp+1q-1}
 &\ \ \ \  \varphi(\mathcal{T}(v;p,q),x)-\varphi(\mathcal{T}(v;p+1,q-1),x)\nonumber\\
  &=x^{(r-2)(q-1)}(\varphi(\mathcal{T}(v;p-q+1,1),x)-\varphi(\mathcal{T}(v;p-q+2, 0),x)).
\end{align}

\noindent
Applying Theorem~\ref{thm_matchingpoly} once more, we have
\begin{align}\label{e_pq01}
 \varphi(\mathcal{T}(v;p-q+1,1),x)=& x^{r-1}\varphi(\mathcal{T}(v;p-q+1,0),x)- x^{r-2}\varphi((\mathcal{T}-v)\cup P_{p-q}^r)
\end{align}
\noindent
and
\begin{align}\label{eq7}
 \varphi(\mathcal{T}(v;p-q+2, 0),x)=& x^{r-1}\varphi(\mathcal{T}(v;p-q+1,0),x)- x^{r-2}\varphi(\mathcal{T}(v;p-q,0)).
\end{align}
 Substituting \eqref{e_pq01} and \eqref{eq7} into \eqref{e_pqp+1q-1}, we obtain
\begin{align*}
  \varphi(\mathcal{T}(v;p,q),x)-\varphi(\mathcal{T}(v;p+1,q-1),x)
 =x^{q(r-2)}(\varphi(\mathcal{T}(v;p-q,0))-\varphi((\mathcal{T}-v)\cup P_{p-q}^r)).
\end{align*}
Note  that $(\mathcal{T}-v)\cup P_{p-q}^r$ is a proper partial hypergraph of $\mathcal{T}(v;p-q,0)$.  Applying Theorems    \ref{YANGYANG-SIAM-10} and \ref{thm_subgraphmatchpoly}, we get the desired result.
\end{proof}

Suppose that $\mathcal{T}$  is an  $r$-uniform supertree (with at least two edges) and $u$ and $v$ are two  vertices incident with an edge $e$ in  $\mathcal{T}$.   Let $\mathcal{T}^{(1)}(u,v;p,q)$ (see Fig.~1(b)) be obtained by attaching two pendent paths of length $p$ and $q$ at $u$ and $v$, respectively.
\begin{theorem}\label{thm_graftat2adjvertex}
 If $p\geq q\geq 1$, then
\begin{align*}
 \mathcal{T}^{(1)}(u,v;p,q)\succ\mathcal{T}^{(1)}(u,v;p+1,q-1).
\end{align*}
In particularly,
\begin{align*}
\rho(\mathcal{T}^{(1)}(u,v;p,q))>\rho(\mathcal{T}^{(1)}(u,v;p+1,q-1)).
\end{align*}

\end{theorem}
\begin{proof}
Using the similar argument as in the proof  of Theorem~\ref{thm_graftat1vertex},   we have
\begin{align}\label{equ10}
  &\varphi(\mathcal{T}^{(1)}(u,v;p,q),x)-\varphi(\mathcal{T}^{(1)}(u,v;p+1,q-1),x)\nonumber\\
  &=x^{(r-2)(q-1)}(\varphi(\mathcal{T}^{(1)}(u,v;p-q+1,1),x)-\varphi(\mathcal{T}^{(1)}(u,v;p-q+2,0),x))\nonumber\\
 &=x^{(r-2)(q-1)}(x^{r-2}\varphi(\mathcal{T}(u;p-q,0),x)-\varphi((\mathcal{T}-v)(u;p-q+1,0),x)).
\end{align}

Let $H_1$  and $H_2$ be the components of $\mathcal{T}\setminus e$
containing vertex $u$ and $v$ respectively, and  $H$ be the union of the remaining components. We denote $H'$ as  the partial hypergraph of $H$ obtained from $H$ by removing $r-2$ vertices contained in $e$.

We may assume that $E(H)\cup E(H_2)$ is not  empty. Otherwise, $\mathcal{T}^{(1)}(u,v;p,q)$ is isomorphic to $H_1(u;p,q+1)$. The result follows from Theorem~\ref{thm_graftat1vertex}.

 When $p=q\geq 1$, applying (b) of Theorem~\ref{thm_matchingpoly} to $\mathcal{T}(u;0,0)$ and edge $e$, we have
  \begin{align}\label{equ12}
   \varphi(\mathcal{T}(u;0,0),x)
=\varphi(H_1\cup H\cup H_2,x)-\varphi((H_1-u)\cup H'\cup (H_2-v),x)
\end{align}

 Similarly,  applying (b) of Theorem~\ref{thm_matchingpoly} to $(\mathcal{T}-v)(u;1,0)$ and the pendent edge  attached at $u$, we have
   \begin{align}\label{equ13}
   &\varphi((\mathcal{T}-v)(u;1,0),x)\nonumber\\
&=x^{r-1}\varphi(H_1\cup H\cup (H_2-v),x)-\varphi((H_1-u)\cup H\cup (H_2-v),x)
\end{align}
   Substituting \eqref{equ12} and \eqref{equ13}  into \eqref{equ10}, we obtain
  \begin{align}\label{e2verp=q}
  &\varphi(\mathcal{T}^{(1)}(u,v;p,q),x)-\varphi(\mathcal{T}^{(1)}(u,v;p+1,q-1),x)\nonumber\\
&=x^{q(r-2)}[\varphi(H_1\cup H\cup H_2,x)-x\varphi(H_1\cup H\cup (H_2-v),x)]\nonumber\\
&\ \ \ \ +x^{(q-1)(r-2)}[\varphi((H_1-u)\cup H\cup (H_2-v),x)-x^{r-2}\varphi((H_1-u)\cup H'\cup (H_2-v),x)]\nonumber\\
&=x^{q(r-2)}[\varphi(H_1\cup H\cup H_2,x)-\varphi(H_1\cup H\cup (H_2-v)\cup\{v\},x)]+x^{(q-1)(r-2)}[\varphi((H_1-u)\cup H\nonumber\\
&\ \ \ \ \cup (H_2-v),x)-\varphi((H_1-u)\cup H'\cup (V(e)-\{u,v\})\cup (H_2-v),x)]
\end{align}

Since $E(H)\cup E(H_2)\not=\emptyset$, either $(H_2-v)\cup\{v\} $ is a proper partial hypergraph of $H_2$, or $H'\cup (V(e)-\{u,v\})$   is a proper partial hypergraph of $H$. By Theorems   \ref{YANGYANG-SIAM-10}, \ref{thm_subgraphmatchpoly} and (\ref{e2verp=q}), the result follows.

  When $p>q\geq1$, applying (b) of Theorem~\ref{thm_matchingpoly} to $\mathcal{T}(u;p-q,0)$ and the edge $e$, we have
  \begin{align}\label{equ15}
\varphi(\mathcal{T}(u;p-q,0),x)=&\varphi(H_1(u; p-q,0)\cup H\cup H_2)\nonumber\\
&-x^{r-2}\varphi((H_1-u)\cup P_{p-q-1}^r\cup H'\cup (H_2-v))
\end{align}
Similarly, applying (b) of Theorem~\ref{thm_matchingpoly} to $(\mathcal{T}-v)(u;p-q+1,0)$ and the pendent edge of the pendent path of length $p-q+1$ attached at $u$, we have
\begin{align}\label{e2adj_vertex_p-q}
\varphi((\mathcal{T}-v)(u;p-q+1,0),x)=&x^{r-1}\varphi(H_1(u; p-q,0)\cup H\cup (H_2-v),x)\nonumber\\
&-x^{r-2}\varphi(H_1(u; p-q-1,0)\cup H\cup (H_2-v),x).
\end{align}
 Substituting \eqref{equ15} and  ~\eqref{e2adj_vertex_p-q} into ~\eqref{equ10} yields
 \begin{align}\label{e2adjvertex3cases}
  &\varphi(\mathcal{T}^{(1)}(u,v;p,q),x)-\varphi(\mathcal{T}^{(1)}(u,v;p+1,q-1),x)\nonumber\\
 &=x^{q(r-1)}\varphi(H_1(u; p-q,0),x)\varphi(H,x)[\varphi(H_2,x)-x\varphi(H_2-v),x)]\\
&+x^{q(r-1)}\varphi(H_2-v,x)[\varphi(H_1(u; p-q-1,0)\cup H,x)-\varphi((H_1-u)\cup P_{p-q-1}^r\cup H'\cup  N_{r-2},x)].\nonumber
\end{align}

We consider the following two cases depending on whether or not $E(H_1)\cup E(H_2)$ is empty.

\noindent
\textbf{Case 1.}  $E(H_1)\cup E(H_2)\neq\emptyset$. Without loss of generality, we assume that $E(H_1)\neq\emptyset$. It is easily seen that
$(H_1-u)\cup P_{p-q-1}^r$ is a proper  partial hypergraph of $H_1(u; p-q-1,0)$.  By Theorems   \ref{YANGYANG-SIAM-10}, \ref{thm_subgraphmatchpoly} and (\ref{e2adjvertex3cases}), we prove the desired result.

\noindent
 \textbf{Case 2.}  $E(H_1)\cup E(H_2)=\emptyset$.
  Since $E(H_1)$ is empty,   $H_1(u; p-q-1,0)$ and $(H_1-u)\cup P_{p-q-1}^r$ are equal to   $P_{p-q-1}^r$.  So $(H_1-u)\cup P_{p-q-1}^r\cup H'\cup  N_{r-2}$ is proper partial hypergraph of $H_1(u; p-q-1,0)\cup H$. By Theorems   \ref{YANGYANG-SIAM-10}, \ref{thm_subgraphmatchpoly} and (\ref{e2adjvertex3cases}),  desired result follows.   \end{proof}

\begin{figure}[!hbpt]
\begin{center}
\includegraphics[scale=0.6]{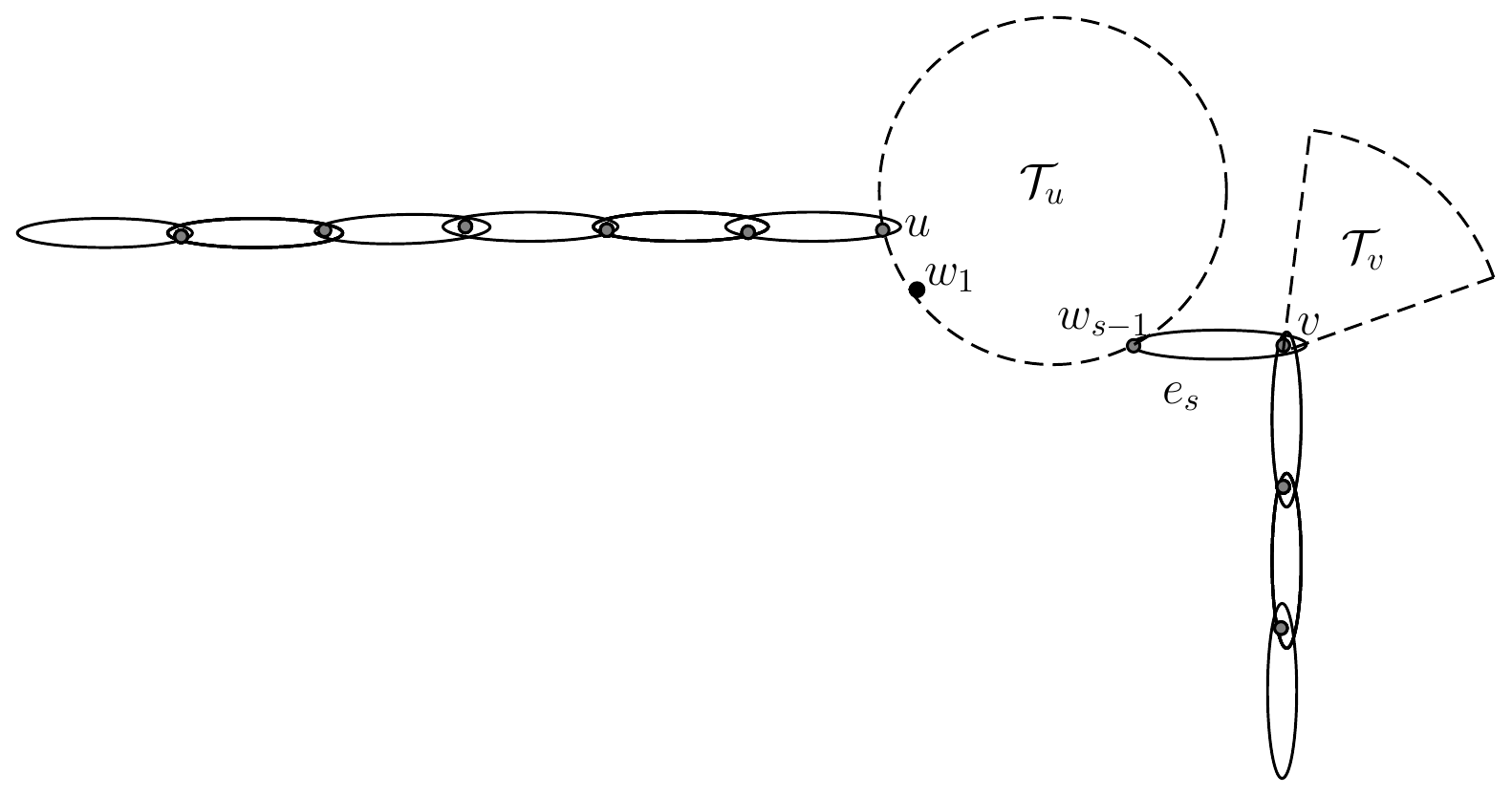}
\caption{Supertree $\mathcal{T}^{(s)}(u,v;6,3)$}.
\end{center}\label{Ts63}
\end{figure}
Suppose that $\mathcal{T}$  is an $r$-uniform supertree  and $u$ and $v$ are two  vertices connected by a  path $P$ of length $s$ in  $\mathcal{T}$, say $P=(u,e_1,w_1,e_2,w_2,\ldots, e_{s-1},w_{s-1},e_s, v)$, and    all the $r-2$ vertices in the set $e_i\setminus\{w_{i-1}, w_i\}$ are of
degree one in $\mathcal{T}$ for $i=2,\ldots,s$, where $w_s=v$.  Let $\mathcal{T}^{(s)}(u,v;p,q)$ be obtained by attaching two pendent paths of length $p$ and $q$ at $u$ and $v$ respectively (see Fig.~2).

\begin{theorem}\label{thm_graftat2vertex-dist-s}
 If $p-q\geq s\geq 1$ and $q\geq1$, then
\begin{align*}
 \mathcal{T}^{(s)}(u,v;p,q)\succ\mathcal{T}^{(s)}(u,v;p+1,q-1).
\end{align*}
In particularly,
\begin{align*}
\rho(\mathcal{T}^{(s)}(u,v;p,q))>\rho(\mathcal{T}^{(s)}(u,v;p+1,q-1)).
\end{align*}
\end{theorem}

\begin{proof}
We  proceed by induction on $s$. For the case $s=1$, the assertion holds by Theorem \ref{thm_graftat2adjvertex}. Let $\mathcal{T}_u$ and $\mathcal{T}_v$ denote the components of $\mathcal{T}\setminus e_s$ containing $u$ and $v$, respectively.
Using the similar argument as in the proof  of Theorem~\ref{thm_graftat2adjvertex},   we have
\begin{align}\label{equ18}
&\varphi(\mathcal{T}^{(s)}(u,v;p,q),x)-\varphi(\mathcal{T}^{(s)}(u,v;p+1,q-1),x)\nonumber\\
&=x^{(q-1)(r-2)}[\varphi(\mathcal{T}^{(s)}(u,v;p-q+1,1),x)-\varphi(\mathcal{T}(u;p-q+2, 0),x)]\nonumber\\
&=x^{(q-1)(r-2)}[x^{r-2}\varphi(\mathcal{T}(u;p-q,0),x)-\varphi((\mathcal{T}-v)(u;p-q+1,0),x)]\nonumber\\
&=x^{q(r-2)}[\varphi(\mathcal{T}(u;p-q,0),x)-\varphi(\mathcal{T}^{(s-1)}_u(u,w_{s-1};p-q+1, 0)\cup(\mathcal{T}_v-v),x)],
\end{align}
where the last equality follows from  $(\mathcal{T}-v)(u;p-q+1, 0)\cong\mathcal{T}^{(s-1)}_{u}(u,w_{s-1};p-q+1,0)\cup (\mathcal{T}_v-v)\cup N_{r-2}$.

Applying  (c) of Theorem~\ref{thm_matchingpoly} to  $\mathcal{T}(u;p-q,0)$ and the  edges  incident to $v$ in $\mathcal{T}_v$, we have
\begin{align}\label{equ19}
\varphi(\mathcal{T}(u;p-q,0),x)=&\varphi(\mathcal{T}_v-v,x)\varphi(\mathcal{T}^{(s-1)}_u(u,w_{s-1};p-q,1),x)-\nonumber\\
& x^{r-2} \varphi(\mathcal{T}_u(u;p-q,0),x)\sum_{e\in E_v\cap E(\mathcal{T}_v)}\varphi(\mathcal{T}_v-V(e),x).
\end{align}
 Substituting \eqref{equ19} into \eqref{equ18}, we obtain
\begin{align}\label{e-Tspq}
&\varphi(\mathcal{T}^{(s)}(u,v;p,q),x)-\varphi(\mathcal{T}^{(s)}(u,v;p+1,q-1),x)\nonumber\\
&=x^{q(r-2)}\varphi(\mathcal{T}_v-v,x)[\varphi(\mathcal{T}^{(s-1)}_u(u,w_{s-1};p-q,1),x)-\varphi(\mathcal{T}^{(s-1)}_u(u,w_{s-1};p-q+1, 0),x)]\nonumber\\
&\quad -x^{(q+1)(r-2)} \varphi(\mathcal{T}_u(u;p-q,0),x)\sum_{e\in E_v\cap E(\mathcal{T}_v)}\varphi(\mathcal{T}_v-V(e),x).
\end{align}
By  induction hypothesis, $\mathcal{T}^{(s-1)}_u(u,w_{s-1};p-q,1)\succ\mathcal{T}^{(s-1)}_u(u,w_{s-1};p-q+1, 0)$. Combining this with Theorems \ref{YANGYANG-SIAM-10} and \ref{thm_subgraphmatchpoly}, we prove the theorem. \end{proof}

\begin{lemma}\label{lem-edgerelease}
Let $\mathcal{T}'$ be  an  $r$-uniform supertree obtained by edge-releasing a non-pendent edge  of $\mathcal{T}$.  Then $\mathcal{T}'$ is a uniform supertree and  $\mathcal{T}\prec \mathcal{T}'$.
\end{lemma}
\begin{proof}
That $\mathcal{T}'$ is a uniform supertree has been proved in \cite{LiShaoQi16}.    $\mathcal{T}$ may be regarded as one  consisting of $s\geq2$ supertrees, say $H_1,\ldots,H_s$, attached at vertices $v_1,\ldots,v_s$ of  $e$, respectively.   It suffices to prove the assertion for $s=2$. Let $H_1\cdot H_2$ be the coalescence of $H_1$ and $H_2$   obtained by identifying   $v_1$ of $H_1$ and $v_2$ of $H_2$.
It is not difficult to verify that $\mathcal{T}'\setminus e\cong H_1\cdot H_2\cup N_{r-1}$ and $\mathcal{T}'-V(e)\cong (H_1-v_1)\cup(H_2-v_2)$.
By Theorem~\ref{thm_matchingpoly}, we have
\begin{align}\label{e1-edgerelease-twosupertree}
\varphi(\mathcal{T},x)&=\varphi(\mathcal{T}\setminus e,x)-\varphi(\mathcal{T}-V(e),x)\nonumber\\
&=x^{r-2}\varphi(H_1\cup H_2,x)-\varphi((H_1-v_1)\cup(H_2-v_2),x).
\end{align}
and
\begin{align}\label{equ21}
\varphi(\mathcal{T}',x)&=\varphi(\mathcal{T}'\setminus e,x)-\varphi(\mathcal{T}'-V(e),x)\nonumber\\
&=x^{r-1}\varphi(H_1\cdot H_2,x)-\varphi((H_1-v_1)\cup(H_2-v_2),x).
\end{align}
By \eqref{e1-edgerelease-twosupertree} and \eqref{equ21}, we deduce that
\begin{align}\label{equ23}
\varphi(\mathcal{T},x)-\varphi(\mathcal{T}',x)=x^{r-2}[\varphi(H_1\cup H_2,x)-x\varphi(H_1\cdot H_2,x)].
\end{align}
Applying (c) of Theorem~\ref{thm_matchingpoly} to $H_1\cup H_2$ and edges in $H_1$ incident to $v_1$, we have
\begin{align}\label{equ24}
 \varphi(H_1\cup H_2,x)&=x\varphi((H_1-v_1)\cup H_2,x)-\varphi(H_2,x)\sum_{e_i\in E_{v_1}\cap E(H_1)}\varphi(H_1-V(e_i),x).
\end{align}
Similarly,
\begin{align}\label{equ25}
\varphi(H_1\cdot H_2,x)&=\varphi((H_1-v_1)\cup H_2,x)-\varphi(H_2-v_2,x)\sum_{e_i\in E_{v_1}\cap E(H_1)}\varphi(H_1-V(e_i),x).
\end{align}
\noindent
Substituting  \eqref{equ24} and   \eqref{equ25}   into \eqref{equ23}, we obtain
\begin{align*}
\varphi(\mathcal{T},x)-\varphi(\mathcal{T}',x)
=x^{r-2}\sum_{e_i\in E_{v_1}\cap E(H_1)}\varphi(H_1-V(e_i),x)[x\varphi(H_2-v_2,x)-\varphi(H_2,x)].
\end{align*}
By Theorems ~\ref{YANGYANG-SIAM-10} and ~\ref{thm_subgraphmatchpoly}, we have $\varphi(\mathcal{T},x)-\varphi(\mathcal{T}',x)>0$ if $x\geq\rho(\mathcal{T})$, so $\mathcal{T}\prec \mathcal{T}'$ holds. \end{proof}

As an application of Theorems \ref{thm_graftat1vertex} and \ref{thm_graftat2adjvertex}, the minimal supertree can be characterized as follows. Note that the upper bound and the extremal supertree have been obtained in \cite{LiShaoQi16}, and they are listed here  for completeness.

\begin{theorem}{\rm (\cite{LiShaoQi16})}\label{thm-hyperstar-loosepath}
If $\mathcal{T}$ is an $r$-uniform supertree with  $m$ edges, then
\begin{equation}\label{e1}
P_m^r\preceq\mathcal{T}\preceq S_{m}^r
\end{equation}
and
 \begin{equation}\label{e2}
\left(2\cos\frac{\pi}{m+2}\right)^{2/r}\leq \rho(\mathcal{T})\leq m^{1/r},
\end{equation}
with left equality in Eq.~\eqref{e1} and Eq.~\eqref{e2} if and only if  $\mathcal{T}\cong P_{m}^r$ and right equality in Eq.~\eqref{e1} and Eq.~\eqref{e2} if and only if  $\mathcal{T}\cong S_{m}^r$.
\end{theorem}

\begin{figure}[h]
\begin{center}
\includegraphics[scale=0.8]{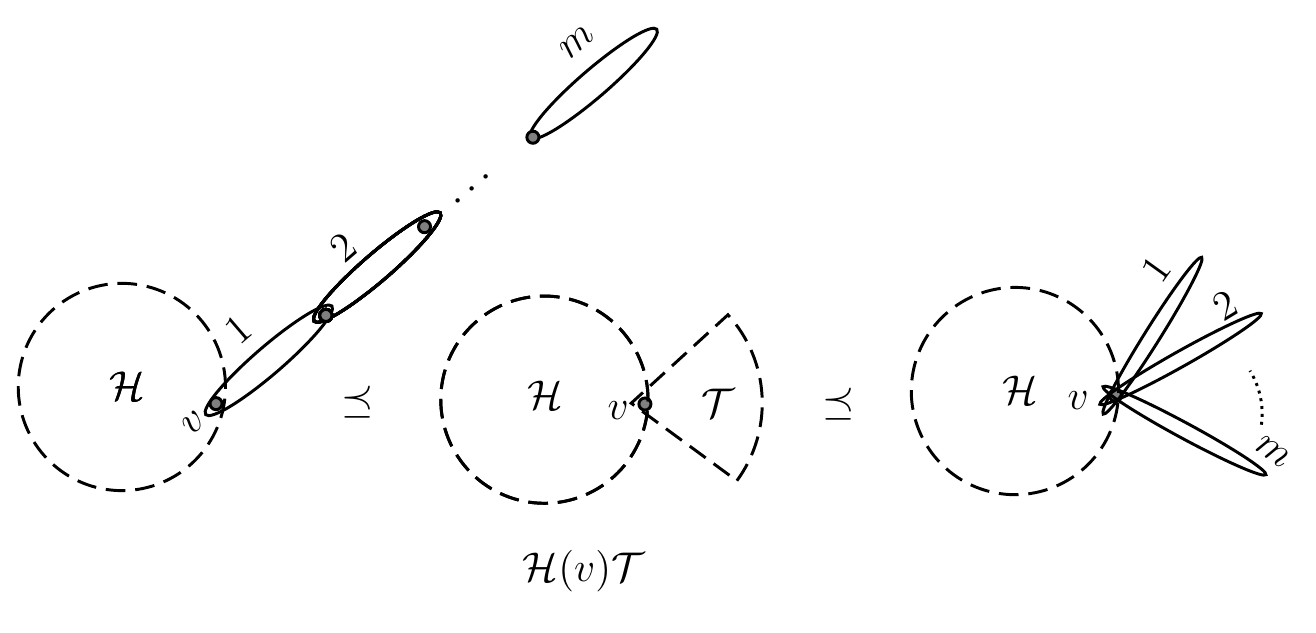}
\caption{Supertree $\mathcal{H}(v)\mathcal{T}$, $\mathcal{T}$ with $m$ edges}.
\end{center}\label{figHPTS}
\end{figure}

Actually, using Theorems \ref{thm_graftat1vertex} and \ref{thm_graftat2adjvertex}, we can deduce the following more general   result.

\begin{theorem}\label{thm-pendenttree-star-path}
Let  $\mathcal{H}$ be an  $r$-uniform supertree, and $v$ a non-isolated vertex of $\mathcal{H}$. Let  $\mathcal{H}(v)\mathcal{T}$ denote the supertree  obtained from $\mathcal{H}$ together with an attached supertree  $\mathcal{T}$ at $v$ of $\mathcal{H}$, see Fig.~3. Then
$$
\mathcal{H}(v)P_m^r\preceq \mathcal{H}(v)\mathcal{T}\preceq \mathcal{H}(v)S_m^r,
$$
where the left--hand side equality holds if and only if $\mathcal{T} \cong P_m^r$ with $v$ as its end vertex
whereas the right--hand side equality holds if and only if $\mathcal{T} \cong S_m^r$ with $v$ as its center\,.
\end{theorem}

\section{Extremal supertrees with given diameter}

Let $S(m,d,r)$ be the set of $r$-uniform supertrees with $m$ edges and diameter $d$. Xiao et. al \cite{XiaoWangDu-supertree-diam-18} determined  the  first two  largest spectral radii of  supertress in $S(m,d,r)$. In this section, we determine the first $\lfloor\frac{d}{2}\rfloor+1$ largest spectral radii of supertrees in $S(m,d,r)$
by using  edge-grafting operations and comparing matching polynomials of supertrees.

Let $\mathcal{H}$ be an $r$-uniform hypergraph and $u$  a vertex of $\mathcal{H}$.
 Let $P_d^r=(v_1,e_1,v_2,e_2,\ldots,  e_d,v_{d+1})$ be a loose path of length $d$.
Denote by $P_d^r(v_i,u)\mathcal{H}$ and $P_d^r(e_j,u)\mathcal{H}$ the hypergraphs   obtained  by identifying vertex $u$ of $\mathcal{H}$ with  vertex $v_i$ of $P_d^r$ and  a core vertex   of $P_d^r$ in $e_j$  respectively (see Fig.~4).

\begin{figure}[h]
\begin{center}
\includegraphics[scale=0.7]{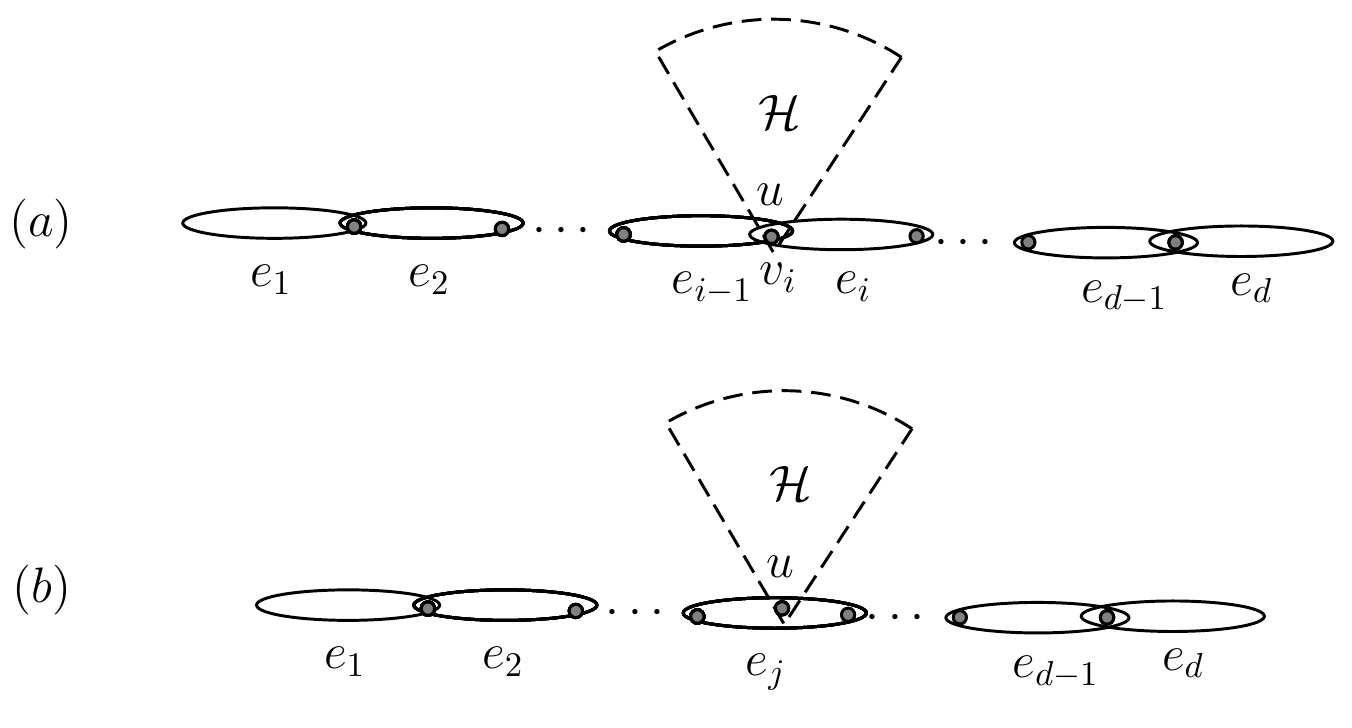}
\caption{Supertrees  (a)  $P_d^r(v_i,u)\mathcal{H}$; (b)    $P_d^r(e_j,u)\mathcal{H}$}.
\end{center}\label{figpdh}
\end{figure}

As an immediate application of Theorems~\ref{thm_graftat1vertex} and \ref{thm_graftat2adjvertex}, we have the following result.
\begin{theorem}\label{thm_verver-edgeedge-veredge}
Let $\mathcal{T}$ be an $r$-uniform supertree,  $r\geq3$. Then
\begin{enumerate}
  \item $\rho(P_d^r(v_{i},u)\mathcal{T})>\rho(P_d^r(v_{j},u)\mathcal{T})$,  if  $2\leq j<i\leq \lfloor d/2\rfloor+1$;
   \item  $\rho(P_d^r(e_{i},u)\mathcal{T})>\rho(P_d^r(e_{j},u)\mathcal{T})$,  if $2\leq j<i\leq \lceil d/2\rceil$;
    \item $\rho(P_d^r(e_{i},u)\mathcal{T})<\rho(P_d^r(v_{i},u)\mathcal{T})$, if $i=2,3,\ldots,d$;
    \item $\rho(P_d^r(e_{i},u)\mathcal{T})<\rho(P_d^r(v_{i+1},u)\mathcal{T})$, if $i=2,3,\ldots,d-1$.
\end{enumerate}
\end{theorem}
\begin{proof}
Note that  $P_d^r(v_{i},u)\mathcal{T}$ and $P_d^r(e_{i},u)\mathcal{T}$ can  be  depicted  as   $\mathcal{T}(u;i-1, d-i+1)$ and    $(\mathcal{T}')^{(1)}(v_i,v_{i+1};i-1, d-i)$ respectively, where $\mathcal{T}'$ denotes the supertree consists of  $\mathcal{T}$ and $e_i$.
The first two assertions follow directly from Theorem~\ref{thm_graftat1vertex} and Theorem \ref{thm_graftat2adjvertex} respectively.

Let $H_1$ and $H_2$ denote the two supertrees obtained from $P_d^r(e_{i},u)\mathcal{T}$  by   moving all edges in  $E_u\cap E(\mathcal{T})$ from $u$ to $v_i$ and   moving the edge $e_{i-1}$ from $v_i$ to $u$, respectively. By Lemma~\ref{lem-edgemoving}, we have $\rho(P_d^r(e_{i},u)\mathcal{T})<\max\{\rho(H_1), \rho(H_2)\}$. However, $H_1\cong H_2\cong P_d^r(v_{i},u)\mathcal{T}$  and assertion (c) holds. Using the similar approach, we can show  the last assertion holds.
\end{proof}

In fact, the last two assertions in Theorem~\ref{thm_verver-edgeedge-veredge} can be generalized as follows.

\begin{theorem}\label{thm_path-verH<edgeH}
Let $\mathcal{T}$ be an $r$-uniform supertree and $P_d^r$ be a loose path of length $d$, with  $d\geq3$ and $r\geq3$. Then  for any $2\leq i\leq d$,  we have
  \[
   P_d^r(e_{\lceil\frac{d}{2}\rceil},u)\mathcal{T}\prec P_d^r(v_{i},u)\mathcal{T}.
  \]
\end{theorem}

\begin{proof}
Suppose that $e_1,e_2,\ldots,e_s$ are all edges  incident with vertex $u$ in $\mathcal{T}$. Applying (c) of Theorem~\ref{thm_matchingpoly} to $P_d^r(e_{\lceil\frac{d}{2}\rceil},u)\mathcal{T}$ and edges $e_1,e_2,\ldots,e_s$, we have
\begin{align*}
   \varphi(P_d^r(e_{\lceil\frac{d}{2}\rceil},u)\mathcal{T}, x)=& \varphi(P_d^r)\varphi(\mathcal{T}-u, x)-x^{r-3}\varphi(P_{\lfloor\frac{d}{2}\rfloor}^r)\varphi(P_{\lceil\frac{d}{2}\rceil-1}^r)\sum_{i=1}^s\varphi(\mathcal{T}-V(e_i), x).
\end{align*}
Similarly, \begin{align*}
   \varphi(P_d^r(v_{i},u)\mathcal{T}, x) =& \varphi(P_d^r)\varphi(\mathcal{T}-u, x)-x^{2r-4}\varphi(P_{i-2}^r)\varphi(P_{d-i}^r)\sum_{i=1}^s\varphi(\mathcal{T}-V(e_i), x).
\end{align*}
 Then
\begin{align}\label{eq27}
  & \varphi(P_d^r(e_{\lceil\frac{d}{2}\rceil},u)\mathcal{T}, x)-\varphi(P_d^r(v_{i},u)\mathcal{T}, x)\nonumber\\
   &=x^{r-3}\sum_{i=1}^s\varphi(\mathcal{T}-V(e_i), x)[x^{r-1}\varphi(P_{i-2}^r\cup P_{d-i}^r,x)-\varphi(P_{\lfloor\frac{d}{2}\rfloor}^r\cup P_{\lceil\frac{d}{2}\rceil-1}^r,x)].
\end{align}

It is easy to see that $P_{i-2}^r\cup N_{r-1}\cup P_{d-i}^r\prec P_{i-1}^r\cup P_{d-i}^r$ as  $P_{i-2}^r\cup N_{r-1}$ is a proper partial hypergraph of  $P_{i-1}^r$. Meanwhile, by Theorem \ref{prop_powertree-pa-pb-matchpoly}, we have $P_{i-1}^r\cup P_{d-i}^r\preceq P_{\lfloor\frac{d}{2}\rfloor}^r\cup P_{\lceil\frac{d}{2}\rceil-1}^r$. Then by Theorems ~\ref{YANGYANG-SIAM-10}, ~\ref{thm_subgraphmatchpoly} and \eqref{eq27},
we have $\varphi(P_d^r(e_{\lceil\frac{d}{2}\rceil},u)\mathcal{T}, x)-\varphi(P_d^r(v_{i},u)\mathcal{T}, x)>0$ if $x\geq\rho(P_d^r(e_{\lceil\frac{d}{2}\rceil},u)\mathcal{T})$. The proof is finished.
\end{proof}

For convenience, we  adopt the notation  from \cite{GuoShao-tree-diam-06}.
Let $m,d,i$ be integers with $2\leq i\leq d\leq m-1$, and $T_{m,d}$ be the set of trees of size $m$ and diameter $d$.
We use $P=(v_1,e_1,v_2,e_2,\ldots, v_d, e_d, v_{d+1})$ to denote the   path  of length $d$.

 Let $T_{(m,d)}(i)$ be the tree on $m$ edges (with diameter $d$) obtained from the path  $P$ by attaching $m-d$ new pendent edges  to the vertex $v_i$.
Let $ \acute{T}_{(m,d)}=\{T_{(m,d)}(i): i=2,3,\ldots, d\}$.

Let $m,d,i,j$ be integers with $2\leq i\neq j\leq d\leq m-2$.  Let $T_{(m,d)}(i,j)$  be the tree  on $m$ edges  (with diameter $d$) obtained from the path  $P$   by attaching  $m-d-1$ new pendent edges to the vertex $v_{i}$ and a new pendent edge to $v_j$, respectively. Let $T''=T_{(m,d)}(\lceil\frac{d}{2}\rceil,\lceil\frac{d}{2}\rceil+1)$.

\begin{lemma}{\rm (\cite{GuoShao-tree-diam-06})}\label{lem-T''}
  For any tree $T\in T_{m,d}\setminus\{\acute{T}_{(m,d)}\cup T''\}$ with  $m\geq d+3\geq6$,  we have $ \rho(T) <\rho(T'')$.
\end{lemma}

The following   results were obtained in \cite{GuoShao-tree-diam-06} and we shall extend  these results from trees to supertrees in this section.

\begin{theorem}{\rm (\cite{GuoShao-tree-diam-06})}\label{thm_tree-diam-order}
(a) The first  $\lfloor\frac{d}{2}\rfloor+1$ spectral radii of  trees in the set  $T_{(m,d)}$ with $m\geq d+3$ and $d\geq3$ are
$
T_{(m,d)}(\lfloor\frac{d}{2}\rfloor+1), T_{(m,d)}(\lfloor\frac{d}{2}\rfloor),\ldots,T_{(m,d)}(3),T_{(m,d)}(2), T''.
$

  (b) The first  $\lfloor\frac{d}{2}\rfloor-1$ spectral radii of  trees in the set  $T_{(m,d)}$ with $m= d+2$ and $d\geq4$ are
$
T_{(m,d)}(\lfloor\frac{d}{2}\rfloor+1), T_{(m,d)}(\lfloor\frac{d}{2}\rfloor),\ldots,T_{(m,d)}(3).
$
\end{theorem}

\begin{figure}[h]
\begin{center}
\includegraphics[scale=0.7]{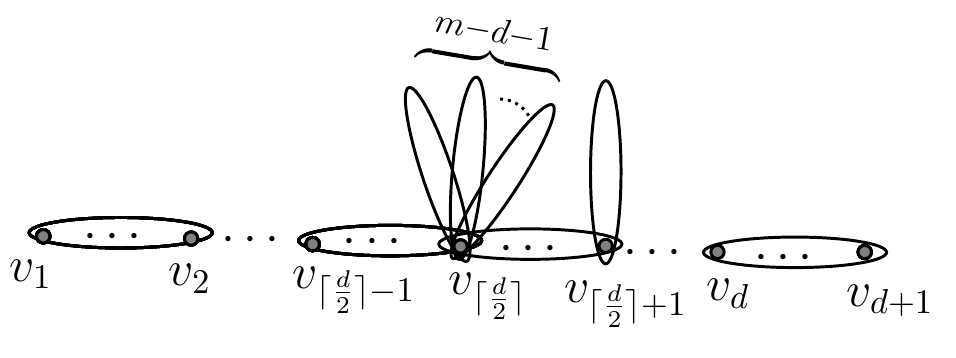}
\caption{Supertree     $ \mathcal{T}''=T_{(m,d)}^r(\lceil d/2\rceil,\lceil d/2\rceil+1)$}.
\end{center}\label{figP2prime}
\end{figure}

Let $m,d,i$ be integers with $2\leq i\leq d-1\leq m$.  Let  $T_{(m,d, r)}(i)$  be a supertree  $P_d^r(e_{i},u)\mathcal{T}$, where $T$ is a hyperstar with $m-d$ edges and $u$ as its center. Note that  $T_{(m,d,r)}(i)\cong T_{(m,d,r)}(d-i+2)$ ($2\leq i\leq d-1$).

Let $ \acute{T}_{(m,d)}^r=\{T_{(m,d)}^r(i): i=2,3,\ldots, d\}$ consisting of the $r$th power of $T_{(m,d)}(i)$ for $i=2,3,\ldots, d$, and let $\mathcal{T}'':=T_{(m,d)}^r(\lceil\frac{d}{2}\rceil, \lceil\frac{d}{2}\rceil+1)$ (see Fig.~\ref{figP2prime}).

\begin{lemma}\label{lem-T''-TmdEdge}
For any  $m\geq d+2\geq5$ and $r\geq3$, we have
  \[
  \rho(\mathcal{T}'')>\rho(T_{(m,d,r)}(\lceil d/2\rceil)).
  \]
\end{lemma}

\begin{proof}
 For simplicity, let $a=m-d-1$ and $b=\lceil\frac{d}{2}\rceil$.
Applying (c) of Theorem~\ref{thm_matchingpoly}   to $\mathcal{T}''$ and    $m-d-1$ pendent edges attached at $v_{\lceil\frac{d}{2}\rceil}$, we have
\begin{align}\label{equ27}
\varphi(\mathcal{T}'', x)&=x^{a(r-1)}\varphi(T^r_{(d+1,d)}(b+1), x)-ax^{(a+1)(r-1)-2}\varphi(P_{b-2}^r\cup (P_{d-b+1}^r, x).
\end{align}
Applying (c) of Theorem~\ref{thm_matchingpoly}   to $T^r_{(d+1,d)}(b+1)$ and  the  pendent edge attached at $v_{\lceil\frac{d}{2}\rceil+1}$, we get
\begin{align}\label{equ28}
\varphi(T^r_{(d+1,d)}(b+1), x)&=x^{r-1}\varphi(P_d^r, x)-x^{2(r-2)}\varphi(P_{b-1}^r\cup P_{d-b-1}^r, x)
\end{align}

\noindent
Substituting \eqref{equ28} into \eqref{equ27}, we deduce
\begin{align}\label{equ29}
\varphi(\mathcal{T}'', x)&=x^{a(r-1)}\varphi(T^r_{(d+1,d)}(b+1), x)-ax^{(a+1)(r-1)-2}\varphi(P_{b-2}^r\cup P_{d-b+1}^r, x)\nonumber\\
&=x^{(a+1)(r-1)}\varphi(P_d^r, x)-x^{(a+2)(r-1)-2}\varphi(P_{b-1}^r\cup P_{d-b-1}^r, x)\nonumber\\
&-ax^{(a+1)(r-1)-2}\varphi(P_{b-2}^r\cup P_{d-b+1}^r, x)
\end{align}
Similarly,
\begin{align}\label{equ30}
\varphi(T_{(m,d,r)}(\lceil d/2\rceil), x)&=x^{(a+1)(r-1)}\varphi(P_d^r, x)-(a+1)x^{(a+1)(r-1)-2}\varphi(P_{b-1}^r\cup P_{d-b}^r, x)
\end{align}

\noindent
By \eqref{equ29} and \eqref{equ30}, we have
\begin{align}\label{e-H1-H2}
&\varphi(\mathcal{T}'', x)-\varphi(T_{(m,d,r)}(\lceil d/2\rceil), x)\nonumber\\
&=x^{(a+1)(r-1)-2}[\varphi(P_{b-1}^r\cup P_{d-b}^r, x)-x^{r-1}\varphi(P_{b-1}^r\cup P_{d-b-1}^r, x)]\nonumber\\
&+ax^{(a+1)(r-1)-2}[\varphi(P_{b-1}^r\cup P_{d-b}^r, x)-\varphi(P_{b-2}^r\cup P_{d-b+1}^r, x)].
\end{align}
Obviously,
\begin{equation*}\label{e-Pd-i-pd-i-1}
  P_{b-1}^r\cup P_{d-b}^r \succ P_{b-1}^r\cup P_{d-b-1}\cup N_{r-1}
\end{equation*}
 as $P_{d-b-1}^r\cup N_{r-1}$ is a proper partial hypergraph of $P_{d-b}^r$. Meantime, by Proposition~\ref{prop_powertree-pa-pb-matchpoly},   we have
 \begin{equation*}\label{e-pi-1pd-i+pi-2pd-i+1}
   P_{b-1}^r\cup P_{d-b}^r=P_{\lceil\frac{d}{2}\rceil-1}^r\cup P_{\lfloor\frac{d}{2}\rfloor}^r\succ P_{\lceil\frac{d}{2}\rceil-2}^r\cup P_{\lfloor\frac{d}{2}\rfloor+1}^r = P_{b-2}^r\cup P_{d-b+1}^r.
 \end{equation*}
Therefore, by Theorems ~\ref{YANGYANG-SIAM-10}, ~\ref{thm_subgraphmatchpoly} and \eqref{e-H1-H2}, $\varphi(\mathcal{T}'',x )<\varphi(T_{(m,d,r)}(\lceil d/2\rceil),x)$ if $x\geq \rho(T_{(m,d,r)}(\lceil\frac{d}{2}\rceil))$.
Consequently, $ \rho(\mathcal{T}'')>\rho(T_{(m,d,r)}(\lceil\frac{d}{2}\rceil))$.
\end{proof}

\begin{lemma}\label{lem_diam-order-outside-1class}
For any  $\mathcal{T}\in S(m,d,r)\setminus\{\acute{T}_{(m,d)}^r\cup \mathcal{T}''\}$   with $m\geq d+3$ and $d\geq3$, we have
 \[
\rho(\mathcal{T})<\rho(\mathcal{T}'').
  \]
\end{lemma}

\begin{proof}
Choose a supertree  $\mathcal{T}\in S(m,d,r)\setminus\{\acute{T}_{(m,d)}^r\cup \mathcal{T}''\}$ with the maximum spectral radius. Let $P^r_m=(v_1,e_1,v_2,e_2,\ldots, v_d, e_d, v_{d+1})$ be the longest loose path in $\mathcal{T}$. Then $\mathcal{T}\setminus \{e_1,\ldots, e_d\}$ is disconnected. Let
$\mathcal{T}_{1},\ldots, \mathcal{T}_{k}$ be the connected components of $\mathcal{T}\setminus \{e_1,\ldots, e_d\}$ which are not isolated vertex.
By Theorem~\ref{thm-pendenttree-star-path} and  the maximality of $\mathcal{T}$, $\mathcal{T}_{j}$ is a hyperstar with a vertex (say $w_{j}$) of the path $P^r_m$ as its center, for $j=1,\ldots,k$. We distinguish two cases according to $w_i\, (i=1,\ldots,k)$ are contained in $\{v_2,\ldots,v_d\}$ or not.

\noindent
\textbf{Case 1.} $\{w_1,\ldots, w_k\}\subseteq \{v_2,\ldots,v_d\}$.

\noindent
Then  $\mathcal{T}$ must be an $r$th power of a tree $T$ of diameter $d$ and size $m$, and $T\in T_{m,d}\setminus\{\acute{T}_{(m,d)}\cup T''\}$. From Lemma~\ref{lem-T''} and Lemma~\ref{lem_eiggraph-eigpowergraph}, it follows immediately that
$$\rho(\mathcal{T})=\rho(T)^{2/r}<\rho(T'')^{2/r}=\rho(\mathcal{T}'').$$

\noindent
\textbf{Case 2.} $\{w_1,\ldots, w_k\}\not\subseteq \{v_2,\ldots,v_d\}$.

\noindent
If $k=1$, then $w_1$ is a vertex of an edge $e_i$,  and $w_1\not\in \{v_i, v_{i+1}\}$, where $2\leq i\leq d-1$. Then $\mathcal{T}\cong T_{(m,d, r)}(i)$. By Theorem~\ref{thm_verver-edgeedge-veredge} and Lemma~\ref{lem-T''-TmdEdge}, we have
 $$\rho(\mathcal{T})\leq \rho(T_{(m,d, r)}(\lceil d/2\rceil))<\rho(\mathcal{T}'').$$

\noindent
If $k\geq 2$, without loss of generality, we may assume $w_1$ is a vertex of edge $e_i$ and   $w_1\not\in \{v_i, v_{i+1}\}$. Denote by $H_1$ and $H_2$ the  supertrees obtained from $\mathcal{T}$ by moving all edges in $E_{w_1}\cap E(\mathcal{T}_1)$ from $w_1$ to $v_i$ and $v_{i+1}$, respectively.
By Theorem~\ref{thm_verver-edgeedge-veredge}, $\rho(\mathcal{T})<\mbox{min}\{\rho(H_1), \rho(H_2)\}$. The maximality of $\rho(\mathcal{T})$ implies that
$H_1, H_2\in \{\acute{T}_{(m,d)}^r\cup \mathcal{T}''\}$ and one of them is $\mathcal{T}''$. So $\rho(\mathcal{T})<\rho(\mathcal{T}'').$ The proof is finished.
\end{proof}

By Theorems \ref{lem_eiggraph-eigpowergraph}, ~\ref{thm_tree-diam-order} and  Lemma~\ref{lem_diam-order-outside-1class}, we have the following results.
\begin{theorem}\label{thm_superteee-diam}
The first  $\lfloor\frac{d}{2}\rfloor+1$ largest spectral radii of supertrees in the set  $S(m,d,r)$ with $m\geq d+3$ and $d\geq3$ are
$T^r_{(m,d)}(\lfloor d/2\rfloor+1), T^r_{(m,d)}(\lfloor d/2\rfloor),\ldots,T^r_{(m,d)}(3),T^r_{(m,d)}(2), \mathcal{T}''.$
\end{theorem}

\begin{theorem}\label{thm_supertree-diam-d+2edges}
The first  $\lfloor\frac{d}{2}\rfloor-1$ largest spectral radii of supertrees in the set  $S(m,d,r)$ with $m=d+2$ and $d\geq4$ are
$T^r_{(m,d)}(\lfloor d/2\rfloor+1), T^r_{(m,d)}(\lfloor d/2\rfloor),\ldots,T^r_{(m,d)}(3).$
\end{theorem}

\section{The second minimal supertree}

Let $P_{m-1}^r=(v_1,e_1,v_2,e_2,\ldots,e_{m-1},v_m)$ be a   loose path of length $m-1$. Denote by $D_{m,r}$ the supertree obtained from $P_{m-1}^r$ by attaching a pendent edge at a core vertex of $e_2$ (see Fig.~\ref{fig-2PATH}(a)). Let $\grave{P}_m^r$ be the supertree obtained from $P_{m-1}^r$ by attaching a pendent edge at the vertex $v_2$ (see Fig.~\ref{fig-2PATH}(b)). We use $S(m, r)$ to denote the set of $r$-uniform supertrees with $m$ edges.

\begin{figure}[h]
\begin{center}
\includegraphics[scale=0.6]{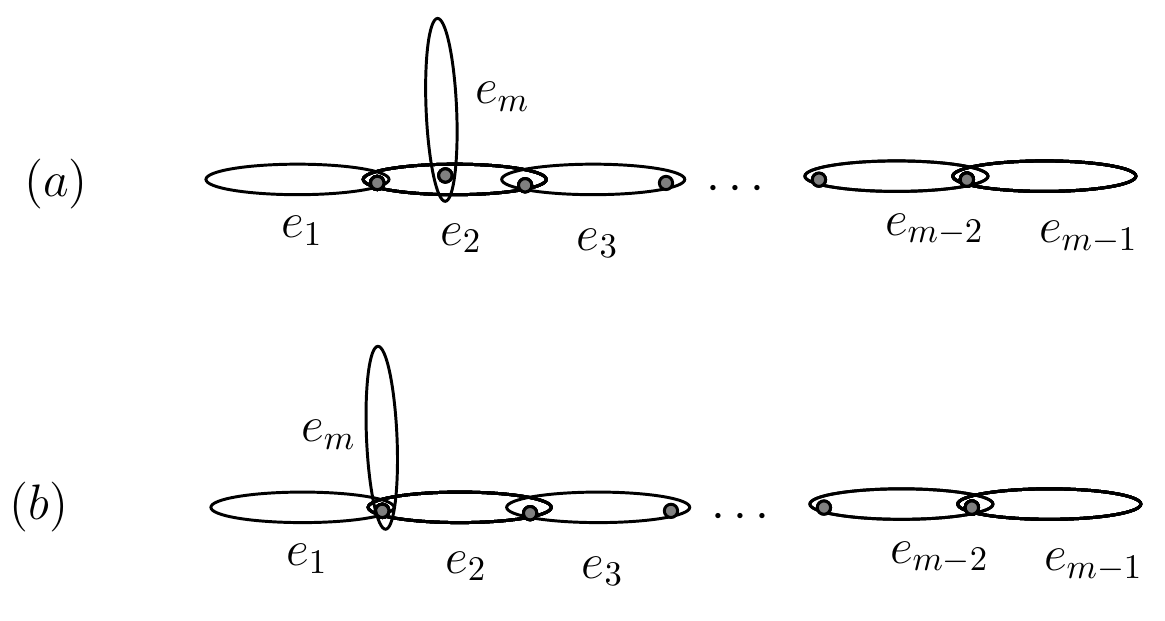}
\caption{Supertrees $(a)\,\,D_{m,r}$;\quad$(b)\,\,\grave{P}_m^r$}.
\end{center}\label{fig-2PATH}
\end{figure}

\begin{theorem}\label{thm-2smallest-superrtree}
Any $r$-uniform  supertree  $\mathcal{T}$  with   $m$ ($m\geq4$) edges different from $P_m^r$  satisfies $\mathcal{T}\succeq D_{m,r}$.
\end{theorem}

\begin{proof}
Choose a supertree $\mathcal{T}_0$ from $ S(m,r)\setminus \{P_m^r\}$ such that $\mathcal{T}_0\preceq \mathcal{T}$ for any $\mathcal{T}\in S(m,r)\setminus \{P_m^r\}$. Then $\mathcal{T}_0$ either has a vertex of degree more than two or has an edge with at least three intersection vertices. We consider the two cases as follows.

\noindent
\textbf{Case 1.}  There exists a vertex of degree greater than two, say $v\in V(\mathcal{T}_0)$ with $\deg(v)\geq3$. Thus $\mathcal{T}_0$ can be described as a supertree  in the form of some supertrees, say $\mathcal{T}_1,\ldots, \mathcal{T}_s$ ($s\geq3$),  attached at a single vertex $v$. Denoted $\mathcal{T}_0$ by $\mathcal{T}_1(v)\mathcal{T}_2(v)\cdots(v)\mathcal{T}_s$ (see Fig.~7(a)). Assume that $\mathcal{T}_i$ has $m_i$ edges for $i=1,\ldots,s$. Let $m'=m-(m_1+m_2)$.
By Theorems ~\ref{thm-pendenttree-star-path} and ~\ref{thm_graftat1vertex}, we have
\begin{equation*}
\mathcal{T}_1(v)\mathcal{T}_2(v)\cdots(v)\mathcal{T}_s\succeq P_{m_1}^r(v) P_{m_2}^r(v)\cdots(v)P_{m_s}^r\succeq P_{m_1}^r(v) P_{m_2}^t(v)P_{m'}^r\succeq P_{1}^r(v) P_{1}^r(v)P_{m-2}^r,
\end{equation*}
where all loose paths $P_{m'}^r$, $P_{m-2}^r$ and $P_{m_j}^r$ ($j=1,\ldots, s$) have $v$ as its end vertex. By the minimality of $\mathcal{T}_0$,  $\mathcal{T}_0=P_{1}^r(v) P_{1}^r(v)P_{m-2}^r= \grave{P}_m^r$.

\begin{figure}[!hbpt]
\begin{center}
\includegraphics[scale=0.7]{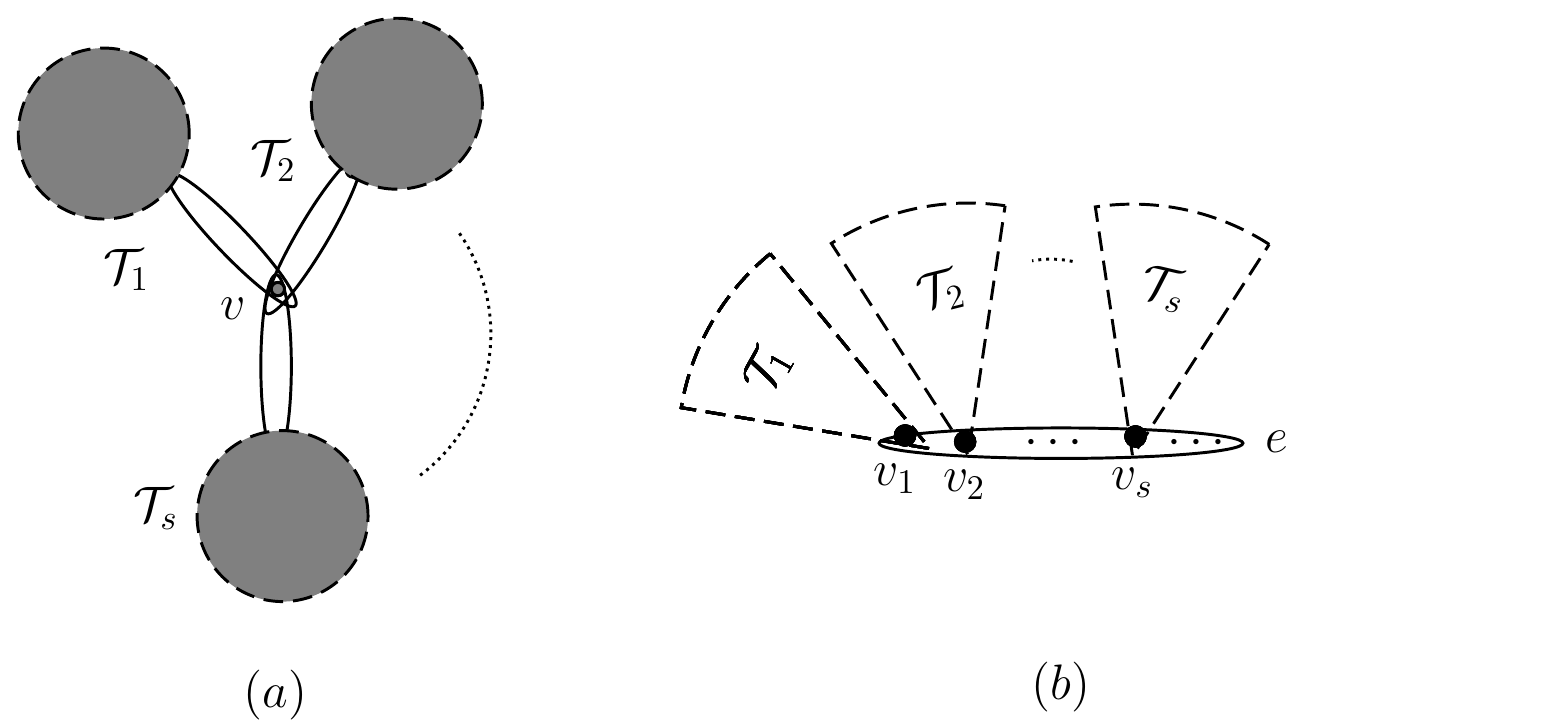}
\caption{Supertrees      $(a)$ and $(b)$}.
\end{center}\label{figT1Ts}
\end{figure}

\noindent
\textbf{Case 2.} There exists an edge $e$ of $\mathcal{T}_0$ with at least three intersection vertices. Without loss of generality, assume that $e=\{v_1,\ldots,v_t\}$ and $\deg(v_i)\geq2$ for $i=1,\ldots,s$ ($3\leq s\leq t$), and other vertices in $e$ (if there are) are core vertices (see Fig.~7(b)).
Then $\mathcal{T}_0$ may be viewed as obtained by attaching supertrees, say $\mathcal{T}_1,\ldots, \mathcal{T}_s$, at $v_1,\ldots,v_s$ respectively.

By Theorems~\ref{thm_graftat2adjvertex}, \ref{thm-pendenttree-star-path} and the minimality of $\mathcal{T}_0$, the following conclusions hold.

(1) $\mathcal{T}_1,\ldots, \mathcal{T}_s$ are pendent paths attached at $v_1,\ldots,v_s$ respectively.

(2) $s=3$.

(3) Two of $\mathcal{T}_1, \mathcal{T}_2, \mathcal{T}_3$ are of length one.

\noindent
Therefore, $\mathcal{T}= D_{m,r}$.

Combining two cases above, we have shown that $\mathcal{T}_0\in \{\grave{P}_m^r,D_{m,r}\}$. Further by (c) of  Theorem~\ref{thm_verver-edgeedge-veredge}, we have  $\grave{P}_m^r\succ D_{m,r}$. So $\mathcal{T}_0=D_{m,r}$.
Thus we conclude that for any $\mathcal{T}\in S(m,r)\setminus\{P_m^r\}$,  $\mathcal{T}\succeq D_{m,r}$.
\end{proof}

\begin{theorem}\label{thm-2smallest-hypergraphs}
The first two smallest spectral radii of   supertrees with $m$ ($m\geq4$) edges are  $P_m^r, D_{m,r}$.
\end{theorem}
\section{Closing remarks}

We conclude this section with some remarks on matching polynomial of a supertree.
The work in this paper is based on the relation between the roots of matching polynomial of a supertree  and its spectrum developed in \cite{Zhang_17}. Using the recurrence relations of matching polynomial of supertrees, the effect on the spectral perturbation of supertree by grafting edges in various situations can be explained.  The methods are initially used to compare spectral radii of supertrees in this paper. The methods  are shown to be efficient in dealing with extremal supertrees with respect to their spectral radii, such as in finding the first two smallest supertrees  and the first several  largest supertrees with given diameter.

For the corresponding problem on  a hypergraph,  the characteristic polynomial of adjacency tensor of a  hypergraph might be used to  compare spectral radii of hypergraphs.
However,  the degree of characteristic polynomial of a  hypergraph is very high relative to its order, and  very little is known about it up to now. Finally, we pose the following problem.

\begin{problem}
What kind of polynomial should be associated with  a hypergraph satisfying the following conditions:

(1) The roots of the associated polynomial consist of the eigenvalues, especially the spectral radius of the hypergraph.

(2) The coefficients of the polynomial reflect certain   structural information of the hypergraph, such as matching, cyclic structure or something more complicated.
\end{problem}

%\subsection*{Acknowledgement.} The authors are grateful to an anonymous referee whose remarks helped us improving the original manuscript.

\end{document}